\documentclass[a4paper,12pt]{amsart}

\usepackage{amssymb,amsbsy,amsmath,amsfonts,amssymb,amscd}
\usepackage{latexsym}
\usepackage{graphics}
\usepackage{color}
\input xy
\xyoption{all}
\newenvironment{satz}[1][Theorem]{%
\vspace{3ex}\noindent\textbf{#1.}\quad\itshape}{%
\par\vspace{3ex}}

\newcommand\sE{{\mathcal E}}
\newcommand\sF{{\mathcal F}}

\newcommand\sL{{\mathcal L}}
\newcommand\sB{{\mathcal B}}
\newcommand\sN{{\mathcal N}}
\newcommand\sX{{\mathcal X}}
\newcommand\sY{{\mathcal Y}}

                 \newcommand\XX{{\mathfrak X}}

\newcommand\la{\lambda}

\newcommand\Ga{\Gamma}
\newcommand\De{\Delta}

\DeclareMathOperator{\Def}{Def}

\DeclareMathOperator{\dlog}{dlog}

\newcommand{\CC}{\ensuremath{\mathbb{C}}}

\newcommand{\ZZ}{\ensuremath{\mathbb{Z}}}
\newcommand{\QQ}{\ensuremath{\mathbb{Q}}}

\newcommand{\sS}{\ensuremath{\mathcal{S}}}

\newcommand{\hol}{\ensuremath{\mathcal{O}}}

\newcommand{\PP}{\ensuremath{\mathbb{P}}}

\newcommand{\ra}{\ensuremath{\rightarrow}}

\def\eea{\end{eqnarray*}}
\def\bea{\begin{eqnarray*}}

\newcommand\dual{\mathrel{\raise3pt\hbox{$\underline{\mathrm{\thinspace d
\thinspace}}$}}}
\newcommand\qe{\ifhmode\unskip\nobreak\fi\quad $\Box$}       

\def\BOX{\hfill\lower.5\baselineskip\hbox{$\Box$}}

\newtheorem{theo}{Theorem}[section]
\newtheorem{remarkk}[theo]{Remark}
\newenvironment{rem}{\begin{remarkk}\rm}{\end{remarkk}}

\newtheorem{defin}[theo]{Definition}

\newtheorem{prop}[theo] {Proposition}
\newtheorem{cor}[theo]{Corollary}
\newtheorem{lemma}[theo]{Lemma}
\newtheorem{example}[theo]{Example}

\newtheorem{claim}[theo]{Claim}

\newcommand{\X}{\ensuremath{\mathcal{X}}}

\DeclareMathOperator{\Aut}{Aut}

\DeclareMathOperator{\inv}{inv}

\begin{document}

\title[Burniat surfaces III]{ Burniat surfaces III: deformations of
automorphisms and extended Burniat surfaces}
\author{I. Bauer, F. Catanese}

\thanks{The present work took place in the realm of the DFG
Forschergruppe 790 "Classification of algebraic
surfaces and compact complex manifolds".
}

\date{\today}

\maketitle

{\em  Dedicated  to
David Mumford.}
\section*{Introduction}

In the present article we continue our investigation, begun in 
\cite{burniat1} and
\cite{burniat2},  of the connected
components of the moduli space (of minimal surfaces $S$ of general 
type)  which contain the Burniat surfaces.
We also correct an error in \cite{burniat2}.

The  main goals that we achieve in this paper are the following:

\begin{enumerate}
\item
We define the family of extended Burniat surfaces for $K^2_S= 3, {\rm 
resp. } \ 4$, and prove that they are
a deformation of the family of nodal  Burniat surfaces with $K^2_S= 
3, {\rm resp. } \ 4$.
\item
We show that the extended Burniat surfaces with $K^2_S= 4$, together
with the nodal  Burniat surfaces with $K^2_S= 4$, form a set  $\sN \sE  \sB _4$
which is a connected component
of the moduli space: thereby we correct theorem 1.1 of \cite{burniat2}.
\item
We show that the extended Burniat surfaces with $K^2_S= 3$, together
with the nodal  Burniat surfaces with $K^2_S= 3$ form an irreducible 
open set $\sN \sE  \sB _3$
of the moduli space, whose closure  $\overline{\sN \sE  \sB _3}$ 
consists of bidouble covers of
normal cubic surfaces in $\PP^3$ and is shown in section 7 to be 
strictly larger than $\sN \sE  \sB _3$.
\item
We answer a question posed on page 562 of \cite{burniat2}, namely,
the integer $m \geq 2 $ in Theorem 1.1 is indeed $ = + \infty$,
and the local moduli space of nodal Burniat surfaces is smooth.
\item
We point out a truly interesting pathology of the moduli space of
varieties with a group $G$ of automorphisms, which is the reason
of our mistake mentioned above (Murphy's law applies then,
but in a different way than foreseen).

It is the fact that for nodal Burniat surfaces $S$, we have a group $G \cong (\ZZ / 2 \ZZ)^2$
of automorphisms, which is also the group of automorphisms of
the canonical model $X$. But whereas $ \Def (X) = \Def  (X, G)$,
$ \Def(S) \neq  \Def (S, G)$: thus even if all deformations of $S$
have a $G$-action, the local moduli space $\Def (S, G)$ for the pairs
yields a proper subvariety  in the smooth  germ $\Def(S)$.
\end{enumerate}

We refer to \cite{burniat1} and \cite{burniat2} for more details 
concerning the
  investigation of the connected
components of the moduli space containing the Burniat surfaces with 
$K^2_S = 6,5,4,2$.

After the results in the present article what remains to be done in 
order to finish this
investigation is to decide, in the case
$K^2_S= 3$ of tertiary Burniat surfaces, whether the irreducible component
mentioned above is also a connected component, describing in detail
all the surfaces which are in the closure and their local deformations.

In \cite{burniat2} we proved that 3 of the 4 irreducible families of 
Burniat surfaces
with $K^2_S  \geq 4$, i.e., of primary and secondary Burniat surfaces, are a
connected component  of the moduli space
of surfaces of general type.

In this paper we consider only   nodal Burniat surfaces with $K^2_S = 
4,3$, showing that a
general deformation of a nodal Burniat surface with $K^2_S = 4$, resp.
with $K^2_S = 3$, is an extended  Burniat surface, still a
bidouble cover
(through the bicanonical map) of a normal Del Pezzo surface of degree 
4 with one ordinary double point, resp. of a cubic surface with three 
nodes.

The  main results of the present paper are  the following:
\begin{theo}\label{main1}
       1) The subset  $\sN \sE \sB_4$ of the  moduli
space of canonical
surfaces of general type $\mathfrak M^{can}_{1,4}$   given by  the 
union of the open set corresponding to
{\em extended} Burniat surfaces with $K^2_S = 4$ with the 
irreducible closed set parametrizing nodal Burniat surfaces with 
$K^2_S = 4$
is an irreducible connected component, normal, unirational
         of  dimension 3.

Moreover  the base of the Kuranishi family of deformations of any 
such a minimal model $S$ is smooth.

\noindent
       2)  The subset  $\sN \sE \sB_3$ of the  moduli
space of canonical
surfaces of general type $\mathfrak M^{can}_{1, 3}$   corresponding to
{\em extended}  and nodal Burniat surfaces with $K^2_S = 3$
is an irreducible open set, normal, unirational
         of  dimension 4.

Moreover the base of the Kuranishi family of $S$ is smooth.
\end{theo}

A very surprising and new phenomenon occurs for nodal surfaces,
confirming Vakil's
`Murphy's law' philosophy (\cite{murphy}).

To explain what happens for the moduli spaces of extended and nodal 
Burniat surfaces, let us  recall again an
old result due to Burns and Wahl (cf. \cite{burnswahl}).

Let $S$ be a minimal surface of general type and let $X$ be its 
canonical model. Denote by $\Def(S)$,
resp. $\Def(X)$, the base of the Kuranishi family of $S$, resp. of $X$.

Their result explains the relation between $\Def(S)$ and $\Def(X)$.

\begin{satz}[Theorem (Burns - Wahl)]
 Assume that $K_S$ is not ample and let $p:S \ra X$ be the 
canonical morphism.

  Denote by $\mathcal{L}_X$ the space of local deformations of the 
singularities of $X$ and by
$\mathcal{L}_S$ the space of deformations of a neighbourhood of the 
exceptional curves of $p$. Then
$\Def(S)$ is realized as the fibre product associated to the Cartesian diagram

\begin{equation*}
\xymatrix{
\Def(S) \ar[d]\ar[r] & \mathcal{L}_S \cong \CC^{\nu}, \ar[d]^{\lambda} \\
\Def(X) \ar[r] & \mathcal{L}_X \cong \CC^{\nu} ,}
\end{equation*}
where $\nu$ is the number of rational $(-2)$-curves in $S$, and 
$\lambda$ is a Galois covering
with Galois group $W := \oplus_{i=1}^r W_i$, the direct sum of the 
Weyl groups $W_i$ of the singular points of $X$.
\end{satz}

An immediate consequence is the following

\begin{satz}[Corollary (Burns - Wahl)]
 1) $\psi:\Def(S) \ra \Def(X)$ is a finite morphism, in 
particular, $\psi$ is surjective.

\noindent
2) If $\Def(X) \ra \mathcal{L}_X$ is not surjective (i.e., the 
singularities of $X$ cannot be
  smoothened independently by deformations of $X$),
  then $\Def(S)$ is singular.
\end{satz}

Assume now that we have $1 \neq G \leq \Aut(S) = \Aut(X)$.

Then we can consider the space of $G$-invariant local deformations of 
$S$, $\Def(S,G)$, resp. $\Def(X,G)$ of  $X$, and we have
a  natural map $\Def(S,G) \ra \Def(X,G)$.

We indeed show here that, unlike the case for  the corresponding 
morphism of   local deformation spaces,
this map needs not to be surjective;
and, as far as we know, the following result gives the first global 
example of such a  phenomenon.

\begin{theo}\label{path}
  The deformations of nodal Burniat surfaces with $K^2_S =4,3$ to 
extended  Burniat surfaces with $K^2_S =4,3$
yield examples where $\Def(S,(\ZZ/2\ZZ)^2) \ra \Def(X,(\ZZ/2\ZZ)^2)$ 
is not surjective.

Moreover, $\Def(S,(\ZZ/2\ZZ)^2) \subsetneq \Def(S)$, whereas for the 
canonical model we have: $\Def(X,(\ZZ/2\ZZ)^2) = \Def(X)$.

Set $G : = (\ZZ/2\ZZ)^2$.  Then the pairs $(S,G)$, where $S$ is the minimal model
of an  extended or nodal Burniat surface, and one gives an effective 
action of $G$ on $S$ (up to automorphisms of $G$) belong to two distinct deformation
types for  
 $K^2_S =4$ and  to four distinct deformation
types for  
 $K^2_S =3$.   

Instead the pairs $(X,G)$, where $X$ is the canonical model
of an  extended or nodal Burniat surface, and one gives an effective 
action of $G$ on $X$ (up to automorphisms of $G$) belong to only one  deformation
type for  
 $K^2_X =4$, and similarly   for  
 $K^2_X =3$.   
\end{theo}

The above phenomenon can  already be seen locally around a node, as 
it will be explained in section 2.
Our results show that the local pathology
does indeed globalize.

\medskip
Our paper is organized as follows: in section 1 we give the 
definition of {\em extended } Burniat surfaces and
describe the different  branch loci of the bidouble covers for nodal 
Burniat surfaces, respectively for   extended Burniat surfaces.

  In the second chapter we analyse bidouble covers of a nodal 
singularity, explaining the phenomenon of theorem \ref{path} locally.

In the third section we show that nodal Burniat surfaces with 
$K^2_S=4,3$ deform to extended Burniat surfaces with $K^2_S=4,3$.

Sections 4 and 5  are instead devoted to the calculation of $H^1(S, 
\Theta_S)$ for nodal and extended  Burniat surfaces,
and its eigenspaces for the  $G = (\ZZ/2\ZZ)^2$ action.

  In the course of doing this we amend a small mistake in 
\cite{burniat2}, lemma 2.10
  and actually generalize this lemma substantially
in order to make it appropriate  for  our present purposes and also 
applicable in other situations.

In the end we succeed to prove that the subset   $\sN \sE \sB_4$ of the  moduli
space of canonical
surfaces of general type $\mathfrak M^{can}_{1,4}$   corresponding to nodal and
extended  Burniat surfaces with $K^2_S = 4$
is an irreducible open set, normal, unirational
         of  dimension 3 (resp. the subset  $\sN \sE \sB_3$  of the  moduli
space of canonical
surfaces of general type $\mathfrak M^{can}_{1, 3}$   corresponding 
to nodal and
extended  Burniat surfaces with $K^2_S = 3$
is an irreducible open set, normal, unirational
         of  dimension 4).

Section 6  is dedicated to the study of one-parameter limits of 
extended Burniat surfaces with $K^2_S=4$,
showing that the subset  of the  moduli
space of canonical
surfaces of general type $\mathfrak M^{can}_{1,4}$   corresponding to
nodal and extended Burniat surfaces with $K^2_S = 4$ is closed.

In section 7 we give examples of other surfaces which lie in the 
closure of the family of extended Burniat surfaces with $K^2_S = 3$.

In the appendix we give an alternative proof of 3 of the 4 assertions 
of proposition \ref{dimensions},
by other methods which could be of independent interest.

\section{Definition of  extended and nodal Burniat surfaces}

Burniat surfaces are minimal surfaces of general type with $K^2
=6,5,4,3,2$ and $p_g = 0$, which were
constructed in \cite{burniat} as minimal resolutions of singular
bidouble covers (=Galois
covers with group $(\ZZ/2\ZZ)^2$) of the
projective plane branched on 9 lines.

We refer the reader to \cite{burniat2} for their construction, and we
shall adhere   to the notation
introduced there.

\medskip
\noindent Let $P_1, P_2, P_3 \in \PP^2$ be three non collinear points
(which we assume to be the points
$(1:0:0)$, $(0:1:0)$ and $(0:0:1)$), and let $P_4,  \dots , P_{3+m},
\  \ m=2, 3,$
be further (distinct)  points   not lying on the sides of the
triangle with vertices $P_1, P_2, P_3$.

Assume moreover that, for $m=2$, the points  $P_1, P_4, P_5$ are collinear,
while, for $m=3$, we shall moreover assume that also $P_2, P_4, P_6$
and $P_3, P_5, P_6$ are collinear (in particular, no four points are
collinear).

Let's denote by
$\tilde{Y}:=\hat{\PP}^2(P_1, P_2,  \dots , P_{3+m})$ the weak Del
Pezzo surface of
degree $6-m$,  obtained blowing  up  $\PP^2$ in the points $P_1, P_2,
\dots , P_{3+m}$.

Saying that $\tilde{Y}$ is a weak Del Pezzo surface means that the
anticanonical divisor $- K_{ \tilde{Y}}$
is nef and big;
        in our case it is not ample, because of the existence of
(-2)-curves, i.e. curves $N_i \cong \PP^1$,
with $ N_i \cdot K_{ \tilde{Y}} = 0$.

Contracting the (-2)-curves $N_i$ we obtain a normal singular Del Pezzo surface
$Y'$ with $- K_{ Y'}$ very ample.

We denote by $L$ the divisor on $\tilde{Y}$ which is the total
transform of a general line in $\PP^2$, by
$E_i$ the exceptional curve lying over $P_i$,   and by $D_{i,1} $ the
unique effective divisor in $ |L - E_i-
E_{i+1}|$, i.e., the proper transform of the line $y_{i-1} = 0$, side
of the triangle joining the points $P_i, P_{i+1}$.

For $m=2$ we have only one  (-2)-curve $N_1$, such that  $\{N_1 \} =
| L -E_1 - E_4 - E_5|$,
while for $m=3$ we also have the curves $N_2, N_3$ such that
        $\{N_2 \} = | L -E_2 - E_4 - E_6|$,
        $\{N_3 \} = | L -E_3 - E_5 - E_6|$.

        Therefore the anticanonical image of $\tilde{Y}$ is a normal surface
$Y' \subset \PP^{6-m}$
        of degree $6-m$, whose singularities are one node $\nu_1$ (an $A_1$
singularity) in the case $m=2$,
        and three nodes  $\nu_1, \nu_2, \nu_3$ in the case $m=3$ (the
(-2)-curve $N_i$ is the total transform of
        the point $\nu_i$).

        \begin{defin}\label{df}
        1) Define the {\em Burniat divisors for $m=2$} as follows:
        $$
        D_1 \in |L-E_1| + |L-E_1 - E_2| + |L-E_1-E_4-E_5| +E_3,
        $$
        i.e., $D_1 = D_{1,1} + N_1+ C_1$, where $C_1 \in  |L-E_1|$ is
assumed to be irreducible,
        whereas $D_2$, $D_3$ are divisors such that
        $$
        \{D_2\} = |L-E_2-E_3| + |L-E_2 - E_4| + |L-E_2-E_5| +E_1,
        $$
        $$
        \{D_3\} = |L-E_3-E_1| + |L-E_3 - E_4| + |L-E_3-E_5| +E_2.
        $$
\noindent
2) The {\em Burniat divisors for $m=3$} are defined to be the
divisors $D_1, D_2, D_3$ such that:
        $$
        \{D_1\}= |L-E_1-E_2| +  |L-E_1-E_4-E_5| + |L-E_1 - E_6| +E_3,
        $$
        $$
        \{D_2\} = |L-E_2-E_3| + |L-E_2 - E_4-E_6| + |L-E_2-E_5| +E_1,
        $$
        $$
        \{D_3\} = |L-E_3-E_1| + |L-E_3 - E_5 -E_6| + |L-E_3-E_4| +E_2.
        $$
        \noindent
        3) The {\em extended Burniat divisors for $m=2$} are given as follows:
        $$\Delta_1 \in |L-E_1| + |L-E_1-E_2| + E_3,
        $$
        $$
        \Delta_2 \in |L-E_2 - E_4| + |L-E_2-E_5| +|2L -E_2 -E_3-E_4-E_5|,
        $$
where we assume the divisor $\Ga_2 \in |2L -E_2 -E_3-E_4-E_5|$ to be
irreducible;
        and $\Delta_3$ is the divisor such that
        $$\{\Delta_3\} =  |L-E_3-E_1|  + |L-E_3 - E_4| + |L-E_3-E_5| +
|L-E_1-E_4-E_5|+E_2.
        $$
        \noindent
        4) The {\em strictly extended Burniat divisors for $m=3$} are
defined as follows:
        $$
        \Delta_1 \in |L-E_1 - E_6|+|2L - E_1 -E_2-E_5-E_6| + |L-E_2 - E_4-E_6|,
        $$
        $$
        \Delta_2 \in |L-E_2 - E_5|+|2L - E_2 -E_3-E_4-E_5| + |L-E_3 - E_5-E_6|,
        $$
        $$
        \Delta_3 \in |L-E_3 - E_4|+|2L - E_1 -E_3-E_4-E_6| + |L-E_1 - E_4-E_5|.
        $$
We make the similar assumption, for each $\Delta_i$,
        that the strict transform of the  conic passing through four of
the five points is irreducible (e.g., we require
the irreducibility of
$
\Ga_1
\in |2L - E_1 -E_2-E_5-E_6|
$).

        \end{defin}

\begin{rem}\label{difference}

1) Observe that $(D_1 + D_2 + D_3) \in | - 3 K_{\tilde{Y}}|$ is a
reduced normal crossing divisor.

2)  Similarly,  $( \De_1 + \De_2 + \De_3 ) \in | - 3 K_{\tilde{Y}} +
\sum N_i |$
        is a reduced normal crossing divisor.

3-2) On the normal Del Pezzo surface $Y'$, for $m=2$,

\begin{itemize}
\item
$D_1$ yields a conic plus two lines, the same does $\De_1$, and
indeed $D_1 =  \De_1 + N_1$
\item
$D_2$ yields four lines,  $\De_2$ yields a conic plus two lines, and
indeed $\De_2 \equiv D_2 + N_1$
\item
$D_3$ yields four lines, the same does $\De_3$, and indeed $ \De_3 = D_3 + N_1$
\end{itemize}
In particular, if  the conic corresponding to  $\De_2$ specializes to
contain the line corresponding to $E_1$,
        we obtain then $D_2$ subtracting the divisor $N_1 \equiv L-E_1
- E_4-E_5$.

Finally, the four lines of $\De_3$ divide into two groups, i.e.,
we can write  $\De_3 =\De_{3,1} + \De_{3,2} + N_1$ so
that, setting $ \Ga_1 : = C_1$ and writing $\De_i = \Ga_i +  \De'_i 
$, for $i=1,2$, then

        $$ (*) : \\   \De'_i +  \De_{3,i} \equiv -  K_{\tilde{Y}}$$
        $$ (**) : \\   \Ga_1+  \Ga_2 \equiv -  K_{\tilde{Y}}.$$

3-3) On the normal Del Pezzo surface $Y'$, for $m=3$,

        $\De_j$ yields a conic and one line, $D_j$ yields three lines,
        and indeed $  \De_j  \equiv  D_j  - N_j +  N_{j-1} + N_{j+1}$.

In particular, if the conic corresponding to  $\De_j$ specializes to
contain the line corresponding to $E_{j-1}$ (here $j \in \ZZ/3 \ZZ$),
        we obtain $D_2$ subtracting the divisor $ N_{j-1} + N_{j+1}$ and
adding the divisor $N_j$.

4) The divisors $D_i$ enjoy the property (cf. \cite{burniat2}) that
there are divisor classes  $L_i$ such that
$D_{i-1} + D_{i+1} \equiv   2
L_i$.

Hence, in particular, $\De_{i-1} + \De_{i+1} \equiv   2
\Lambda_i$, where, for $m=3$, $ \Lambda_i : = L_i + N_i$. Instead, for
$m=2$, this formula holds only
for $i=1$, and  we set $ \Lambda_j : = L_j $ for $j=2,3$.

5-2) Assume now $m=2$, and that the conic corresponding to $\Ga_2$
becomes reducible:
if the conic passes through $P_1$, then necessarily $\Ga_2$ splits as
$N_1 + E_1 + |L- E_2 - E_3|$, hence the conic is the union of two lines.
If the conic is the union of two lines in another fashion,
then necessarily either $|L- E_2 - E_5|$ or $|L- E_2 - E_4|$ is a  component of
       $\Ga_2$, hence $\De_2$ is not reduced.

5-3) Assume $m=3$ and that one or more of these conics become reducible.
E.g., assume that the conic corresponding to $\Ga_2$ becomes reducible,
and observe that this will be the case if the conic passes through
$P_1$ or $P_6$.
We disregard this degeneration if the corresponding divisor $\De_2$
will be non reduced.
The only possibility left over is that  $\Ga_2$ splits as before,
$N_1 + E_1 + |L- E_2 - E_3|$.
This degeneration will be considered admissible.

\begin{defin}\label{ext}
Assume $m=3$ and that one or more of these conics $\Ga_j$ become reducible
in the admissible way $\Ga_j = N_{j-1} + E_{j-1}  + |L- E_j - E_{j+1}|$
       (here, as usual, $j \in \ZZ/3 \ZZ$).

In this case we define the extended Burniat divisors by subtracting to
$\Ga_j$ the nodal divisor $ N_{j-1}$ it contains, by subtracting again
       the nodal divisor $ N_{j-1}$ to $\De_{j+1}$ and adding it to $\De_{j-1}$.
\end{defin}

\end{rem}

We can now consider (cf. \cite{ms}, \cite{sbc}) the associated
bidouble covers  $S \ra \tilde{Y}$ with branching divisors the  Burniat
divisors, respectively the extended Burniat divisors.

\begin{defin}
          A  {\em secondary nodal Burniat surface} is obtained, for $m=2$,
          as  a bidouble cover  $S \ra \tilde{Y}$ with branch divisors the
three Burniat divisors.

           In the case $m=3$ we obtain  a {\em tertiary nodal Burniat
surface} $S$.

        $S$ is then a minimal surface of general type with $p_g(S) =q(S)=
0$,  $K_S^2 = 6-m$
(cf. \cite{burniat2}).

          If we let the three branch divisors be extended Burniat divisors,
then we obtain
a non minimal surface $S'$ whose minimal model $S$ is called
        a {\em secondary extended  Burniat surface},  respectively
        a  {\em tertiary extended   Burniat surface}.

\end{defin}

\begin{rem} 1) In the nodal Burniat case the surface $S$ does not
have an ample canonical divisor $K_S$,
due to the existence of (-2)-curves, which are exactly the inverse images of
the  (-2)-curves $N_i \subset \tilde{Y}$.

For this reason we call the above Burniat surfaces
        of  {\em nodal type}. We denote their canonical model by $X$, and
observe that
$X$ is a finite bidouble cover of the normal Del Pezzo surface $Y'$.

For $m=2$ $X$  has precisely one node (an
$A_1$-singularity, corresponding to the contraction of the
(-2)-curve) as singularity. While, for $m=3$, $X$ has exactly three nodes as
singularities.

2) In the extended  Burniat case $S'$ is not minimal.
In the strictly extended   Burniat case the
inverse image of each $N_i$
       splits as the union of two disjoint (-1)-curves.
In this latter case $S$ has ample canonical divisor, hence $S=X$.

3) In all cases, the morphism $ X \ra Y'$ is exactly the bicanonical
map of $X$ (see \cite{burniat2}).

4) Nodal Burniat surfaces are parametrized by a family with smooth
base of dimension $2$ for $m=2$,
of dimension $1$ for $m=3$.

        Strictly extended  Burniat surfaces are parametrized by a
family with
smooth base of dimension $3$ for $m=2$,
of dimension $4$ for $m=3$.
\end{rem}

The key feature is that, both for nodal Burniat surfaces, and for
extended  Burniat surfaces,
the canonical model $X$ is a finite bidouble cover of a singular Del
Pezzo surface $Y'$,
which has one node in the case $m=2$, and three nodes for $m=3$ (in
the latter  case $Y'$ is a cubic surface in $\PP^3$).

In both cases the direct image $p_* (\hol_X)$ splits as a direct sum
of four reflexive character sheaves
of generic rank $1$.

In the next section we shall describe how the covering behaves in the
neighbourhood of a node
in the two respective cases, and how these local coverings deform to
each other (the Burniat case deforms
to the extended Burniat case).

\section{Local calculations around the nodes}\label{locdefnode}

In this section we consider finite bidouble covers of a node
of Du Val type, i.e., yielding
singularities which are at worst  RDP 's (rational double points).

We obtain a classification which is a subset of the one made
   in \cite{autRDP}, classifying quotients of RDP's by actions of
$\ZZ /  2 \ZZ$ or of $G = (\ZZ /  2 \ZZ)^2$.

We only need to look at Table 2, page 90, and Table 3, page 93, ibidem,
to see which  quotients of a rational double
point by an involution,
or by a pair of commuting involutions,  yield an $A_1$-singularity, 
i.e., a node.

There are six cases for such coverings of Du Val type of a node $Y$,
which in local
holomorphic coordinates
is given by
$$ xy - z^2 = 0. $$
In order to be more informative in our description, we denote by
$\tilde{Y}$ the resolution of $Y$,
which is the total space of a line bundle on $N \cong \PP^1$ of
degree  $-2$ (hence $N^2 = -2$).
Denoting the bidouble cover of $Y$ by $X$, we shall obtain, through
the normalization of the fibre
product, a finite bidouble cover of $\tilde{Y}$, for which we shall
give the three corresponding
branch divisors.

In the case where $X$ is not irreducible, we shall describe a
connected component $X'$ of $X$.

\begin{enumerate}
\item
$X'$ = $Y$ (the covering is \'etale).
\item
$X' = \CC^2$, $X$ has two components and the covering morphism is given by
$$  (u,v) \mapsto (x=u^2, y = v^2, z = uv).$$
The branch divisor  on $\tilde{Y}$ is just the (-2)-curve $N$.
\item
$X' = \{ w^4 = xy \}$, $X$ has two components and the covering
morphism is given by
$$  (x,y,w) \mapsto (x, y , z = w^2).$$
The branch divisor  on $\tilde{Y}$ consists of the (-2)-curve $N$
plus two fibres; the double cover
of $\tilde{Y}$ has two nodes and resolving them we get the minimal
resolution of the
$A_3$ singularity $X'$.
\item
$X = \{ w^2 = uv \}$ and the covering morphism is given by
$$  (u,v) \mapsto (x=u^2, y = v^2, z = w^2).$$
The three intermediate $\ZZ/2 \ZZ$ covers are the two double covers (2), (3)
described above,
plus the intermediate cover (here $ a : = uw, b : = vw$)
$$\{  (x,y,z,a,b)| Rank \begin{pmatrix}
x , a , z , b \\
a , z, b , y
\end{pmatrix} =  1 \},$$
which is the cone over a rational normal quartic
(set $x = t_0^4, a =t_0^3t_1, z =t_0^2 t_1^2, z =t_0 t_1^3,  z = t_1^4$).

The branch divisors  on $\tilde{Y}$ are two: the (-2)-curve $N$ and the divisor
$D$ formed by two fibres. The three intermediate
double covers depend on the choice of the branch locus: $N$,
respectively $N+D$,
respectively $D$.

\item
$X' =  \{ z^2 = (w^2 +y^{k+1}) \cdot  y  \}$, $X$ has two components
having a singularity of type $D_{k+3}$, and the covering
morphism is given by
$$  (y,z, w) \mapsto (x= w^2 +y^{k+1} , y , z ).$$
The branch divisor on $\tilde{Y}$ is the total transform of the divisor
$C : = \{ x = y^{k+1}, z^2= y^{k+2}\}$ which is irreducible with a 
cusp for $k$ odd,
else it is reducible with a $\frac{k}{2}$-tacnode for $k$ even.

In particular, $N$ is part of the branch locus.

\item
$X = \{ w^2 = (u - v^{k+1}) ( u + v^{k+1}) \} = \{ w^2 = u^2  - v^{2k+2}  \}$
   and the covering morphism is given by
$$  (u,v, w) \mapsto (x=u^2, y = v^2, z = uv).$$
$X$ is a singularity of type $A_{2k+1}$ and, in order to treat a new case,
  we make the assumption $ k \geq 1$.

The three intermediate $\ZZ/2 \ZZ$ covers are the smooth double cover 
(2), the double cover (5)
$ \{ w^2 = x   - y^{k+1}  \}$,
  and a third singularity which we omit to describe.

The branch divisors  on $\tilde{Y}$ are two: the (-2)-curve $N$ and the
the total transform of the divisor   $C$ above.

The three intermediate  covers depend on the choice of the branch locus: $N$,
or  $N+C'$,
or  $C'$, where $C'$ is the strict transform of $C$.

\end{enumerate}

\bigskip

Letting $p : X \ra Y$ be the finite bidouble cover, the direct image
sheaf $p_* \hol_X$
splits as $$\hol_Y \bigoplus ( \oplus_{i=1,2,3} \sL_i),$$ where in
the first case the
reflexive sheaves $ \sL_i$ are locally free.

To describe the other cases we use the reflexive sheaf $\sF$
generated by $u,v$ as
$\hol_Y $-module, with relations
$$y u - z v = 0, z u - xv = 0.  $$
We get

\begin{enumerate}
\setcounter{enumi}{1}
\item $X' = \CC^2$,
$  (u,v) \mapsto (x=u^2, y = v^2, z = uv),$
$$p_* \hol_X = (\hol_Y \oplus \sF)^{\oplus 2}$$
\item
$X' = \{ w^4 = xy \}$,
$  (x,y,w) \mapsto (x, y , z = w^2)$
$$p_* \hol_X = (\hol_Y \oplus \hol_Y )^{\oplus 2}$$
\item
$X = \{ w^2 = uv \}$ ,
$  (u,v) \mapsto (x=u^2, y = v^2, z = w^2)$
$$p_* \hol_X = (\hol_Y \oplus \sF)^{\oplus 2},$$
with generators $ 1, \{u,v\}, w, \{a= wu, b = vw\}$.
\item
$X' = \{ w^2 = x   - y^{k+1} \}$,
$  (y,z, w) \mapsto (x= w^2 +y^{k+1} , y , z ),$
$$p_* \hol_X = (\hol_Y \oplus \hol_Y )^{\oplus 2}.$$
\item
$X = \{  w^2 = u^2  - v^{2k+2}  \}$,
$  (u,v,w) \mapsto (x=u^2, y = v^2, z = uv)$
$$p_* \hol_X = (\hol_Y \oplus \sF)^{\oplus 2}.$$
\end{enumerate}

\begin{rem}
Cases 1, 3 and 5 are the case where we have a flat bidouble cover, i.e.,
$p_* \hol_X$
is locally free. In cases 2, 4 and 6 we have non-flat
        bidouble covers, but with the same character sheaves. We shall
soon show how
case 4 (resp. : case 6)) deforms to case 2.
\end{rem}

\begin{prop}
In case 2) $ X= Spec ( (\hol_Y \oplus \sF) \oplus  (\hol_Y \oplus \sF))$,
where the two addenda are orthogonal, and the algebra structure is
determined by the nondegenerate pairing $ \sF \times \sF \ra \hol_Y$.

In case 4)  $ X= Spec ( (\hol_Y \oplus \sF) \oplus  w(\hol_Y \oplus \sF))$,
and the algebra structure is
determined by the nondegenerate pairing $ \sF \times \sF \ra \hol_Y$,
together with the assignment $ w^2 = z$.

In case 6)  $ X= Spec ( (\hol_Y \oplus \sF) \oplus  w(\hol_Y \oplus \sF))$,
and the algebra structure is
determined by the nondegenerate pairing $ \sF \times \sF \ra \hol_Y$,
together with the assignment $ w^2 = x - y^{k+1}$.
\end{prop}

We omit the simple proof.

Case 4) deforms now to case 2) by changing the assignment  $ w^2 = z$
to  $ w^2 = z + t$, $t\neq 0$, so that $w$ becomes then a local unit
at the origin. Similarly case 6) deforms to case 2).

We can relate the resulting picture with the local semiuniversal deformation of
a node.

\begin{prop}\label{def-act}
Let $t \in \CC$ , and consider the action of $G: = (\ZZ/2 \ZZ)^2$ on $\CC^3$
generated by $\sigma_1(u,v,w) = (u,v,-w)$, $\sigma_2(u,v,w) = (-u,-v,w)$.
Then the hypersurfaces $X_t = \{ (u,v,w)| w^2 = uv + t\}$ are $G$-invariant,
and the quotient $X_t / G $ is the hypersurface
$$ Y_t = Y_0 =  \{ (x,y,z)| z^2 = xy\} ,$$
which has a nodal singularity at the point $x=y=z=0$.

$X_t \ra Y_t$ is a bidouble covering of type 2 for $t\neq 0$, and of type 4
for $t=0$. We get in this way a flat family of (non flat) bidouble covers.

\end{prop}

\begin{proof}
The invariants for the action of  $G$ on $\CC^3 \times \CC$ are:
$$ x: =u^2, y:  = v^2, z : = uv , s: = w^2, t.$$

Hence the family $\XX$  of the hypersurfaces $X_t$ is the inverse image of the
family of hypersurfaces $ s = z +t$ on the product
$$Y' \times \CC^2 = \{x,y,z,s,t)| xy = z^2 \} .$$
Hence the quotient of $X_t$ is isomorphic to $Y'$.

The rest was already explained before.
\end{proof}

\begin{rem}
i) The simplest way to view $X_t$ is to  see $\CC^2$ as a double cover
of $Y'$ branched only at the origin,
and then $X_t$ as a family of double covers of $\CC^2$
branched on the curve $ uv + t = 0$, which acquires a double point for $t=0$.

ii)
The involution $\sigma_3(u,v,w) = (-u,-v,-w)$ has only the origin as
fixed point,
which lies on $X_0$. Whereas $\sigma_3$ acts freely on $X_t$, for $t \neq 0$.

$ Fix (\sigma_1) = \{ w = 0\}$, and $ \{ w = 0\} \cap X_t = \{ uv + t = w=0\}$.

\noindent
Finally, $ Fix (\sigma_2) = \{ u=v = 0\}$, and $ \{  u=v = 0\} \cap
X_t = \{ u=v = 0,w^2=  t \}$,
which consists of two points for $t \neq 0$, one for $t=0$.

The corresponding branch loci are the origin, for $t=0$,
the divisor $ z=0$, and the point $ x=y= z-t = 0$.

iii) If we pull back the bidouble cover $X_t$ to $\tilde{Y}$, and we
normalize it,
we can see that
\begin{itemize}
\item
$D_3$ is, for $t=0$, the nodal curve $N$, and is the empty divisor
for $t\neq 0$;
\item
$D_1$ is, for $t \neq 0$, the inverse image of the curve $z + t = 0$;
while, for
$t=0$, it is only its strict transform, i.e. the divisor $D$
considered previously, made up of
two fibres;
\item
$D_2$ is an empty divisor for $t=0$, and the nodal curve $N$
for $t\neq 0$.

\end{itemize}

\end{rem}
\begin{rem}
Part iii) of the previous remark shows that, as $t \ra 0$, one subtracts
the nodal divisor $N$ to $D_2$, and adds it to $D_3$; while for $D_1$,
it specializes to $ D + N$, and then we subtract $N$.

This is precisely the algorithm which applies when passing from
extended Burniat
to Burniat divisors.
\end{rem}

The really interesting part of the story comes now: the family $X_t$ admits
a simultaneous resolution only after that we perform a base change
$$ t = \tau^2 \Rightarrow  w^2 - \tau^2 = uv.$$

\begin{defin}
Let $\XX \ra T'$ be the family where $$\XX = \{ (u,v,w,\tau ) |  w^2
- \tau^2 = uv \}$$
and $T'$ is the affine line with coordinate $\tau$.

Define $\sS \subset \XX \times \PP^1$ to be one of the small
resolutions of $\XX$,
and $\sS'$ to be the other one, namely:

$$\sS : \{ (u,v,w,\tau)(\xi) \in  \XX \times \PP^1| \
\frac{w-\tau}{u} =   \frac{v}{w+\tau} = \xi \}$$
$$\sS' : \{ (u,v,w,\tau)(\eta) \in  \XX \times \PP^1| \
\frac{w+\tau}{u} =   \frac{v}{w-\tau} = \eta \}.$$

Let $G$ be the group $G \cong (\ZZ/2 \ZZ)^2$ acting on $\XX$ trivially on
the variable
$\tau$, and else as in proposition \ref{def-act}.
Let further $\sigma_4$ act by $\sigma_4(u,v,w,\tau) = (u,v,w,- \tau) $,
       let $G' \cong (\ZZ/2 \ZZ)^3$ be the group generated by $G$ and $\sigma_4$,
       and let $H \cong (\ZZ/2 \ZZ)^2$ be the subgroup $\{ Id, \sigma_2,
\sigma_1\sigma_4 , \sigma_3\sigma_4\}$.
\end{defin}

The following is a rephrasing and a generalization of a discovery of
Atiyah  in our context:
we omit the simple proof.

\begin{lemma}\label{nolift}
The biregular action of $G'$ on $\XX$ lifts only to a birational
action on $\sS$,
respectively $\sS'$. The subgroup $H$ acts on $\sS$, respectively $\sS'$,
as a group of biregular automorphisms.

The elements of
$ G' \setminus H = \{ \sigma_1, \sigma_3, \sigma_4 , \sigma_2\sigma_4\}$
yield isomorphisms between $\sS$ and  $\sS'$.

The group $G$ acts on the punctured family $\sS \setminus \sS_0$,
in particular it acts on each fibre $\sS_{\tau}$.

Since $\sigma_4$ acts trivially on $\sS_0$,
the group $G'$ acts on $\sS_0$ through its direct summand $G$.

The biregular actions of $G$ on $\sS \setminus \sS_0$ and on $ \sS_0$
do not patch together to a biregular action on $\sS$, in particular
$\sigma_1$ and $\sigma_3$ yield birational maps which are
not biregular: they are called Atiyah flops (cf. \cite{atiyah}).
\end{lemma}

\section{Nodal Burniat surfaces deform to  extended  Burniat surfaces}

In this section we shall show, for each value of $m=2,3$,
       that the canonical models $X$
of nodal Burniat surfaces with $K_X^2 = 6-m$, together with the extended
  Burniat surfaces with $K_X^2 = 6-m$ are parametrized by a
family with smooth connected base of respective dimension $1+m$,
which maps  to the moduli space via a finite morphism.

We shall treat first the easier case $m=2$.

\begin{prop}\label{famiglia4}
There exists a family, with connected base
$$B \subset  \{ (C_1 , \Ga_2) | C_1 \in | L - E_1| , \Ga_2 \in |2L -
E_2 - E_3 -E_4 -E_5| \}$$
where $ C_1, \Ga_2   $ are as in Definition \ref{df} ($ C_1$ is
irreducible and either
$\Ga_2   $ is irreducible, or splits as $ N_1 + E_1 + |L-E_2 - E_3|$),
parametrizing a flat family of canonical models, including exactly all
the nodal Burniat surfaces  and the extended  Burniat surfaces
with $K^2_X = 4$.
\end{prop}

\begin{proof}
Recall that in this case $D_1 + D_3 = \De_1 + \De_3$, and that
$N_1$ is a connected component of the above divisor  $D_1 + D_3 =
\De_1 + \De_3$.

We can therefore construct a family of double covers
$$\tilde{W}_b \ra \tilde{Y}  $$
such that the inverse image of $N_1$ is a (-1)-curve.
Blowing down this (-1)-curve we get a family of finite double covers
$W'_b \ra Y'$,
which are nodal and equisingular.

Consider the pull back of the divisors $\De_2$ in the case where
$\Ga_2$ is irreducible, and of the divisors $D_2$ in the case where  $\Ga_2$ is
reducible.

Since $\De_2 \equiv D_2 + N_1$, and the divisor $N_1$ becomes trivial on
$W'_b$, since it contracts to a smooth point, it follows that all
these divisors are
linearly equivalent, and we have a family of divisors on the family  $W'_b$

We consider then the family of double covers $X_b \ra  W'_b$ branched on these
divisors, and on the nodes of $W'_b$.
\end{proof}

\begin{prop}\label{famiglia3}
There exists a family, with connected base
$$T \subset  \{ (\Ga_1 , \Ga_2, \Ga_3) \}$$ where $ \Ga_1 , \Ga_2,
\Ga_3  $ are as in
Definitions
\ref{df} and \ref{ext},
parametrizing a flat family of canonical models, including exactly all
the nodal Burniat surfaces  and the extended  Burniat surfaces
with $K^2_X = 3$.
\end{prop}

\begin{proof}
Given a triple $(\Ga_1 , \Ga_2, \Ga_3)$, according to the reducibility
of each $\Ga_i$, there corresponds either a Burniat divisor, or an extended
Burniat divisor. We take the corresponding bidouble cover of $\tilde{Y}$,
hence  we construct four families of smooth surfaces, which are not necessarily
minimal. We take now the corresponding canonical models, which
are finite bidouble covers of the normal Del Pezzo surface $Y'$.

Observe that, given $p' : \tilde{S} \ra \tilde{Y}$, and $\pi :
\tilde{Y} \ra Y'$,
$$X = Spec ( \pi_* (p')_* \hol_{\tilde{S}})= Spec (\hol_{Y'}
\bigoplus (\oplus_{i=1}^3
\sF_i)).$$

Now the reflexive sheaves $\sF_i$ correspond to Weil divisors on $Y'$,
and they are independent of $t \in T$ by virtue of 4) of remark
\ref{difference}.

The multiplication maps correspond to a family of Weil divisors on $Y'$:
whence we get a flat family on $Y' \setminus Sing (Y')$.
Locally around the nodes the structure of the deformation is as described in
the previous section, therefore the family is flat everywhere.

\end{proof}

Observe that the latter proof works also in the case $m=2$.

\section{A corrigendum to Burniat surfaces II}

Parts 1), 2) and 3 ) of the following lemma were contained in Lemma
2.10 of \cite{burniat2},
while 4) corrects a wrongly stated assertion of 2) of loc. cit.

We also amend the  proof for the correct assertions.

\begin{lemma}\label{Hauptlemma} Consider a finite set of distinct linear forms
$$l_{\alpha} : =  y - c_{\alpha} x,
\alpha \in A$$ vanishing at the origin in $\CC^2$.

      Let $p \colon Z
\ra \CC^2$  be the blow up of
        the origin, let $D_{\alpha} $ be the strict transform of the line
$L _{\alpha} : = \{ l_{\alpha} = 0 \}$, and let $E$ be the exceptional divisor.

Let  $\Omega_{\CC^2}^1 ((\dlog l_{\alpha})_{\alpha \in A})$ be the
sheaf of rational 1-forms $\eta$ generated
by  $\Omega_{\CC^2}^1$ and by the differential forms
$d \log l_{\alpha}$ as an $\hol_{\CC^2}$-module and define similarly
$ \Omega^1_{Z}((\log D_{\alpha})_{\alpha \in A}) $. Then:

\begin{enumerate}
\item
$p_* \Omega^1_{Z}(\log E)(-E) = \Omega_{\CC^2}^1$,
\item
$p_* \Omega^1_{Z}(\log E, (\log D_{\alpha})_{\alpha \in A}) =
\Omega_{\CC^2}^1 ((\dlog l_{\alpha})_{\alpha \in A})$,
\item
$p_* \Omega^1_{Z}((\log D_{\alpha})_{\alpha \in A}) =
        \{ \eta \in \Omega_{\CC^2}^1 ((\dlog l_{\alpha})_{\alpha \in A})|
\eta =$

$=\sum_{\alpha} g_{\alpha} \dlog l_{\alpha} + \omega,
\omega \in \Omega_{\CC^2}^1, \sum_{\alpha} g_{\alpha}(0) = 0\}$.
\item
$ p_*
\Omega^1_{Z}( (\log D_{\alpha})_{\alpha
\in A})(E) \supset $  $\Omega_{\CC^2}^1 ((\dlog l_{\alpha})_{\alpha \in A})$
and $$dim_{\CC}  [ p_*
\Omega^1_{Z}( (\log D_{\alpha})_{\alpha
\in A})(E) / \Omega_{\CC^2}^1 ((\dlog l_{\alpha})_{\alpha \in A}) ]  = d-2$$
is supported  at the origin, where $ d : = |A|$.
More precisely, we have an exact sequence

$$ 0 \ra \Omega^1_{\CC^2} \ra  p_*
\Omega^1_{Z}( (\log D_{\alpha})_{\alpha
\in A})(E)   \ra
\bigoplus_{\alpha=1}^d \hol_{D_\alpha}(0) \ra \CC^2_0 \ra 0.
$$
\item
Assume w.l.o.g. $c_1 = 0$ in the following formulae: then

$p_* \Omega^1_{Z}(\log D_1)(-E) \subset \Omega_{\CC^2}^1(\dlog l_1)$
is the subsheaf of forms
$$\{ \omega = \alpha dx + \beta  \frac{dy}{y}| \beta (0) = 0,
\frac{\partial \beta}{ \partial y}(0) = 0,
\frac{\partial \beta}{ \partial x}(0) + \alpha (0)= 0\}. $$

\item
$p_* \Omega^1_{Z}(-E) = \mathfrak M_0 \Omega_{\CC^2}^1$.
\item
$p_* \Omega^1_{Z}(\log D_1, \log D_2)(-E) \subset 
\Omega_{\CC^2}^1(\dlog l_1, \dlog l_2)$
is the subsheaf of forms
$$\{ \omega = \alpha \frac{dx}{x}+ \beta  \frac{dy}{y}| \alpha (0) =
0, \beta (0) = 0,
\frac{\partial (\alpha + \beta)}{ \partial x}(0)  = 0,
\frac{\partial (\alpha + \beta) }{ \partial y}(0) = 0\}. $$

\end{enumerate}
\end{lemma}

\begin{proof}

We show 2), 3), 4), 5)  and 7).

Observe that
$$ p_* \Omega^1_{Z}( (\log D_{\alpha})_{\alpha \in A})(mE) $$
consists of rational differential 1-forms
$\omega$ which, when restricted to $\CC^2 \setminus \{0\}$, yield sections of
$\Omega_{\CC^2}^1 ((\dlog l_{\alpha})_{\alpha \in A}) $.

Therefore in particular $\omega \prod_{\alpha \in A}  l_{\alpha} $ is
a regular holomorphic 1-form on $\CC^2$. Hence $\omega$,
      modulo holomorphic 1-forms,  can be written as
$$ \omega = \frac{f } {\prod_{\alpha \in A}  l_{\alpha}} dx +
\frac{g } {\prod_{\alpha \in A}  l_{\alpha}} dy,$$ where $f,g$ are
pseudopolynomials of degree in $y$ strictly less than
$d : = {\rm card} (A)$.

      Since  $ dy = d l_{\alpha}  + c_{\alpha} dx $,  the condition that
$\omega$
        restricted to $\CC^2 \setminus \{0\}$ yields a section of
$\Omega_{\CC^2}^1 ((\dlog l_{\alpha})_{\alpha \in A}) $ implies that
$  l_{\alpha} | (f +   c_{\alpha} g) $.

Whence $  l_{\alpha}$ divides $f x + y g$, and we conclude,
since $ {\prod_{\alpha \in A}  l_{\alpha}}$   is a  pseudo polynomial
of degree $d$,  that

$$ f x + y g =   c(x) {\prod_{\alpha \in A}  l_{\alpha}}.$$

This allows us to write,  modulo holomorphic 1-forms,
$$ \omega = \frac{ g ( dy - \frac{y}{x} dx)} {\prod_{\alpha \in A}
l_{\alpha}}  +
\frac{c } {x} dx,$$
where now $ c \in \CC$.

Let us pull back $ \omega$ to $Z$, using local coordinates $ (x,t)$
such that $ y = xt$, and where we make the assumption $c_{\alpha}
\neq 0,  \forall \alpha$.

$$ p^* \omega =  \frac{ x^{-d} g(x, xt) ( x  dt)} {\prod_{\alpha \in
A}  (t - c_{\alpha})}  +
\frac{c } {x} dx.$$

The pull back form has logarithmic poles along $E$ iff $  g(x, y)$
has multiplicity
at least $d-1$ at the origin, and poles of order at most one along $E$
iff $  g(x, y)$ has multiplicity
at least $d-2$ at the origin.

Observe that the $d$ polynomials $ P_{\beta} : =  \prod_{\alpha \in
A,\alpha \neq \beta }  l_{\alpha}$
are linearly independent and homogeneous of degree $d-1$, hence they generate
the vector space of homogeneous polynomials of degree $d-1$, hence
they generate the ideal of holomorphic functions vanishing at the origin of
      multiplicity
at least $d-1$.

Hence  $  g(x, y)$ has multiplicity
at least $d-1$ iff we can write
$$ g = \sum _{\alpha \in A}   g_{\alpha} P_{\alpha}.$$

And since $g$ is a pseudo polynomial of degree  $\leq d-1$, the $g_{\alpha}$'s
are just functions of $x$.

In this case we can write

$$  \omega = \frac{c } {x} dx +  \sum_{\alpha \in A}
\frac{g_{\alpha}}{l_{\alpha}}
(dy - \frac{y}{x} dx)
=  \frac{1 } {x}  [ c dx + \sum_{\alpha \in A}
\frac{g_{\alpha}}{l_{\alpha}}  (x dy - y dx) ] .$$
$$ =  \frac{1 } {x}  [ c dx + \sum_{\alpha \in A}
\frac{g_{\alpha}}{l_{\alpha}}  (x d l_{\alpha}  + xc_{\alpha} dx - y
dx) ]=
\sum_{\alpha \in A} \frac{g_{\alpha}}{l_{\alpha}} d l_{\alpha} +
\frac{1 } {x} dx [c - \sum_{\alpha \in A} g_{\alpha}] .$$

The above  form $  \omega$ does not have poles on the line $x=0$  if
and only if
$c = (\sum_{\alpha \in A} g_{\alpha}(0)) $.

Observing that the strict transform of the line $x=0$ is not among 
the divisors $D_{\alpha}$,
we establish claim (2), while (3) follows since $c=0$ iff
there are no poles along $E$.

The very first assertion of (4) follows by (2), so let's proceed to verify
the other assertions.

Assume now that $ p^* \omega$ has poles of order $1$ along $E$;
equivalently, assume
that  $  g(x, y)$ has multiplicity
at least $d-2$ at the origin. Since we already dealt with the case
where this multiplicity
is   at least $d-1$, we may assume that  $  g(x, y)$ is homogeneous
of degree $d-2$,
and that $c=0$.

Argueing as done before, the space of homogeneous polynomials of degree $d-2$
has as basis the $d-1$ polynomials ($\beta = 1, \dots, d-1$)
      $$ Q_{\beta} : =  \prod_{\alpha \in A,\alpha \neq \beta ,\alpha \neq
d}  l_{\alpha}.$$

Whence $ g = \sum _{\alpha \in A, \alpha \neq d}   g_{\alpha} Q_{\alpha},$
where $g_{\alpha} \in \CC$, and we may write:

$$  \omega =  \sum_{\alpha =1}^{d-1} \frac{g_{\alpha}}{l_{\alpha}
l_d}  (dy - \frac{y}{x} dx)  .$$

Since we want no poles on the line $x=0$, we must have
$$ \sum_{\alpha =1}^{d-1} \frac{g_{\alpha} y }{y^2} = 0
\Leftrightarrow \sum_{\alpha =1}^{d-1} g_{\alpha} = 0.$$

Under this condition we may then write
$$  \omega =  \sum_{\alpha =1}^{d-1} \frac{g_{\alpha}}{l_{\alpha}
l_d}  (d l_{\alpha} )  ,$$
which has logarithmic poles along $ l_{\alpha}  = 0$.

      Logarithmic poles along $ l_d  = 0$ follow  by writing
$$  \omega =  \sum_{\alpha =1}^{d-1} \frac{g_{\alpha}}{l_{\alpha}
l_d}  (d l_d ) +
      \sum_{\alpha =1}^{d-1} \frac{g_{\alpha} (c_d -
c_{\alpha})}{l_{\alpha} l_d}  dx   ,$$

      and observing that $ \sum_{\alpha =1}^{d-1} \frac{g_{\alpha} (c_d -
c_{\alpha})}{ y - c_{\alpha} x}  $
      vanishes for $l_d = 0 $ since on $\{ l_d = 0 \}$ we have $  y = c_d x$.

      Applying the residue sequence, we see that each such form 
$\omega$ has as residue
on $D_{\alpha}$
      a function with a single pole at most at the origin $O$, and with
coefficient of $\frac{1}{x}$
      respectively equal to $r_d : =  \sum_{\alpha =1}^{d-1}
\frac{g_{\alpha}}{ (c_d - c_{\alpha})}$
      in the case of $D_d$, and $ r_{\alpha} : = -  \frac{g_{\alpha}}{
(c_d - c_{\alpha})}$
      in the case of $D_{\alpha}$.

      In other words, the sum of the `double' residues  is $0$, and the
other condition
      $ \sum_{\alpha =1}^{d-1}g_{\alpha}  =0$ can be also written down as
      $ \sum_{\alpha =1}^{d}c _{\alpha}  r _{\alpha}  =0$.

To show 5), observe that

$$p_* \Omega^1_{Z}(\log D_1)(-E) \subset p_* \Omega^1_{Z}(\log D_1)
\subset \Omega_{\CC^2}^1(\dlog l_1).$$

Take coordinates $x,y$ such that $ l_1 = y$, and write
$\omega = \alpha dx + \beta   \frac{dy}{y}$.

We just pull back $\omega$ on the blow up $Z$ in the chart where we have $ y =
tx$, and impose that it
lies in the span of $$ x \frac{dt}{t} , x dx.$$

We have
$$\omega =   \alpha (x, tx) dx + \beta  (x, tx) ( \frac{dt}{t} +
\frac{dx}{x} )$$
and we must clearly have $\beta (0)=0$.

Then $ \beta  (x, tx)  \frac{dt}{t}$ is a multiple of $ x
\frac{dt}{t}$, and it suffices to require
that $ \alpha (x, tx)  +  \frac{1}{x} \beta  (x, tx)$ be divisible by $x$.

Writing $\beta (x,y) = \beta_1 x + \beta_2 y + \dots$, our condition
boils down to the divisibility by $x$
of
$$  \alpha (0) + \beta_1 + \beta_2 t \Leftrightarrow \beta_2  = 0 , \
\alpha (0) + \beta_1 = 0 .$$

Finally, let us show 7).
Write $$ \omega = \alpha \frac{dx}{x}+ \beta  \frac{dy}{y}$$
and pull back to the blow up in the chart where $y = tx$:
we get  $$  (\alpha + \beta) \frac{dx}{x}+ \beta  \frac{dt}{t},$$
which must be divisible by $x$, hence in particular $\beta (0) = 0$.
Looking at the other chart we get symmetrically  $\alpha (0) = 0$.

Now, $\alpha + \beta$ must vanish of order two, in order that
its pull back be divisible by $x^2$.

\end{proof}

Corollary 2.11 of \cite{burniat2} has also to be modified,
as we shall show in proposition \ref{dimensions} of the next section: 
the only non vanishing cohomology
group
$$ H^0 ( \Omega_{\tilde{Y}}^1 (\log (D_i)) (K_{\tilde{Y}} + L_i  ))
     = H^0 ( \Omega_{\tilde{Y}}^1 (\log (D_i)) (E_i - E_{i+2}   ))$$
occurs for $i=3$ (not for $i=1$), and it has dimension equal to $1$.

\section{Local deformations of the extended  Burniat surfaces}

We begin with an easy but useful observation
\begin{lemma}\label{poles}
Assume that $N$ is a  connected component of a smooth divisor $D \subset Y$,
where $Y$ is a smooth projective surface.

Moreover, let $M$ be a divisor on $Y$. Then
$$ H^0 ( \Omega_{Y}^1 (\log (D-N))(N+M ) ) = H^0 ( \Omega_{Y}^1 (\log 
(D))(M) ) $$
provided $(K_Y + 2 N + M ) \cdot N < 0$.
\end{lemma}

\begin{proof}
The cokernel of $ \Omega_{Y}^1 (\log (D)) (M) \ra   \Omega_{Y}^1 (\log
(D-N))(N + M) $ is supported on $N$
and equal to $ \Omega_{N}^1 (N +M) = \hol_N (K_Y + 2 N + M)$.

\end{proof}

The lemma will be applied several times  in the case where $N \cong 
\PP^1$ and $N^2 < 0$.

Another useful lemma which will be crucial in some calculation is the following

\begin{lemma}\label{triangle}
Assume that we have three linearly independent linear forms on $\PP^2$,
$ l_1 := x_1 , l_2 := x_2 , l_3 := x_3$. Then

\begin{enumerate}
\item
$$ H^0 ( \Omega_{\PP^2}^1  (2))$$
has as basis the 3 1-forms, for $ \ j < i $,
$$\eta_{ji} : = x_j d x_i - x_i dx_j = - \eta_{ij}. $$
\item

$$ H^0 ( \Omega_{\PP^2}^1 (\dlog l_1,\dlog l_2, \dlog l_3 ) (1))$$
has as basis the 6 1-forms
$$\omega_{ij} : = \frac{x_j d x_i - x_i dx_j}{x_i}. $$
\item
$$ H^0 ( \Omega_{\PP^2}^1 (\dlog l_1,\dlog l_2, \dlog l_3 ) (2))$$
has as basis the 3 1-forms $\eta_{ji}$, for $ \ j < i $, plus the 6 1-forms
$x_j \omega_{ij}  $
and the 3 1-forms $ x_1 \omega_{23}, x_2 \omega_{31}, x_3 \omega_{12}. $

\end{enumerate}
\end{lemma}

\begin{proof}

1) is well known and follows from the Euler sequence.

2)
Take the chart $x_i \neq 0 \Leftrightarrow x_i = 1$: then in this chart
$\omega_{ij} : =-  dx_j $ is a regular 1-form.

In the chart $x_j = 1$ we have $\omega_{ij} : = \frac{d x_i }{x_i},$
while in the chart  $x_h = 1$ we have $\omega_{ij} : =  x_j \frac{d
x_i }{x_i} - d x_j.$

Hence $\omega_{ij}$ has logarithmic poles on $ x_i = 0$, and the
coefficient of the
logarithmic term vanishes for $x_i = x_j = 0$, and is equal to $1$ in
$x_i = x_h = 0$.

The above observation shows the linear independence of the above 6 forms.

Moreover, $\omega_{ij} $ is an eigenvector with character $\lambda$
for the $\CC^*$-action
$ x_i \mapsto \lambda x_i$, hence $\omega_{ij} \in  H^0 (
\Omega_{\PP^2}^1 (\dlog l_1,\dlog l_2, \dlog l_3 ) (1))$.

It suffices to show that this space has vector dimension equal to 6.

This follows however from the exact sequence
    $$ 0 \ra   \Omega_{\PP^2}^1 (1) \ra  \Omega_{\PP^2}^1 (\dlog 
l_1,\dlog l_2, \dlog l_3 ) (1)
\ra \oplus_{i=1}^3 \hol_{l_i}(1) \ra 0$$
and the vanishing of $H^j ( \Omega_{\PP^2}^1 (1) )$ for $ j=0,1$.

3)  Observe that $\omega_{ij} = \frac{1}{x_i}\eta_{ji} $,
so that  $x_i \omega_{ij} = \eta_{ji} = - \eta_{ij}= x_j \omega_{ji} $.

Moreover, if $h \neq i, j$,  $x_h \omega_{ij} - x_j \omega_{ih} = \eta_{jh}$,
so that the products  $x_r \omega_{ij} $ generate a subspace of
dimension at most 12.

  By the exact sequence

   $$ 0 \ra   \Omega_{\PP^2}^1 (2) \ra  \Omega_{\PP^2}^1 (\dlog 
l_1,\dlog l_2, \dlog l_3 ) (2)
\ra \oplus_{i=1}^3 \hol_{l_i}(2) \ra 0$$
and since $H^1 ( \Omega_{\PP^2}^1 (2) ) = 0$, $h^0 (\hol_{l_i}(2)) = 3$
we infer that the dimension is indeed 12.

Since $H^0 (\hol_{l_i}(2)) $ is generated by $H^0 (\hol_{\PP^2}(1)) 
\otimes_{\CC} H^0 (\hol_{l_i}(1)) $
we conclude that the 12 1-forms are a basis.
\end{proof}
\begin{lemma}\label{cross}
Assume that we have two linearly independent linear forms on $\PP^2$,
$ l_1 := x_1 , l_2 := x_2$.

\begin{enumerate}

\item
$ H^0 ( \Omega_{\PP^2}^1 (\dlog l_1,\dlog l_2) (1))$
has as basis the 4 forms
$$\omega_{ij} : = \frac{x_j d x_i - x_i dx_j}{x_i},  \ 1 \leq i,j 
\leq 3, \ i \neq 3.$$
\item
$ H^0 ( \Omega_{\PP^2}^1 (\dlog l_1,\dlog l_2) (2))$
has as basis the 3 forms $\eta_{ji}$, for $ \ j < i $, plus the 6 forms
$ x_2 \omega_{12}, x_1 \omega_{21}, x_3 \omega_{13}, x_3 \omega_{23}, 
x_2 \omega_{12}, x_1 \omega_{23}. $

\end{enumerate}
\end{lemma}
\begin{proof}
Follows from lemma \ref{triangle} observing that $ H^0 ( 
\Omega_{\PP^2}^1 (\dlog l_1,\dlog l_2) (i))$ is a subspace of
$ H^0 ( \Omega_{\PP^2}^1 (\dlog l_1,\dlog l_2, \dlog l_3) (i))$. The 
above two sets of vectors are linearly independent
and the dimensions are 4, resp. 9.
\end{proof}

\begin{cor}\label{dxi}
1) Let
$\omega  \in H^0 ( \Omega_{\PP^2}^1 (\dlog l_1,\dlog l_2) (1)).$

  Then there are complex numbers $a_{ij}$ such that:
\begin{multline*}
\omega = a_{12}\omega_{12} + a_{21}\omega_{21} +a_{13}\omega_{13} 
+a_{23}\omega_{23} = \frac{dx_1}{x_1} (a_{12}x_2 -a_{21}x_1 
+a_{13}x_3) + \\
+ \frac{dx_2}{x_2} (-a_{12}x_2 +a_{21}x_1 +a_{23}x_3)
+dx_3(-a_{13} -a_{23}).
\end{multline*}
2) Let  $\omega  \in H^0 ( \Omega_{\PP^2}^1 (\dlog l_1,\dlog l_2) 
(2))$: then we can write
\begin{multline*} \omega = a_{12}\eta_{12} + a_{13}\eta_{13} 
+a_{23}\eta_{23} 
+a_{212}x_2\omega_{12}+a_{121}x_1\omega_{21}+a_{313}x_3\omega_{13}+\\
+a_{323}x_3\omega_{23}+a_{213}x_2\omega_{13}+a_{123}x_1\omega_{23} =
\end{multline*}
\begin{multline*}
  = \frac{dx_1}{x_1} (-a_{12}x_1x_2 -a_{13}x_3x_1 
+a_{212}x^2_2-a_{121}x_1^2+a_{313}x_3^2+a_{213}x_2x_3) + \\
+ \frac{dx_2}{x_2} (a_{12}x_1x_2 -a_{23}x_3x_2 
-a_{212}x_2^2+a_{121}x_1^2+a_{323}x_3^2 +a_{123}x_1x_3) +\\
+dx_3(a_{13}x_1 +a_{23}x_2-a_{313}x_3-a_{323}x_3-a_{213}x_2-a_{123}x_1).
\end{multline*}
3)Any  $ \omega \in H^0 ( \Omega_{\PP^2}^1 (\dlog l_1,\dlog l_2, 
\dlog l_3) (1))$ can be written as:
\begin{multline*} \omega = a_{12}\omega_{12} + a_{13}\omega_{13} 
+a_{23}\omega_{23} 
+a_{21}\omega_{21}+a_{31}\omega_{31}+a_{32}\omega_{32}=
\end{multline*}
\begin{multline*}
  = \frac{dx_1}{x_1} (a_{12}x_2 -a_{21}x_1+a_{13}x_3-a_{31}x_1) + \\
+ \frac{dx_2}{x_2} (-a_{12}x_2 +a_{21}x_1 +a_{23}x_3 -a_{32}x_2) +\\
+\frac{dx_3}{x_3}(-a_{13}x_3 +a_{31}x_1+a_{32}x_2-a_{23}x_3).
\end{multline*}
4) Any  $ \omega \in H^0 ( \Omega_{\PP^2}^1 (\dlog l_1,\dlog l_2, 
\dlog l_3) (2))$ can be written as:
\begin{multline*} \omega = a_{12}\eta_{12} + a_{13}\eta_{13} 
+a_{23}\eta_{23} 
+a_{212}x_2\omega_{12}+a_{313}x_3\omega_{13}+a_{323}x_3\omega_{23}+\\
+a_{121}x_1\omega_{21}+a_{131}x_1\omega_{31}+a_{232}x_2\omega_{32} 
+a_{123}x_1\omega_{23}+ a_{231}x_2\omega_{31}+ a_{312}x_3\omega_{12} =
\end{multline*}
\begin{multline*}
  = \frac{dx_1}{x_1} (-a_{12}x_1x_2 -a_{13}x_3x_1 +a_{212}x_2^2+a_{313}x_3^2-\\
-a_{121}x_1^2-a_{131}x_1^2-a_{231}x_2x_1+a_{312}x_3x_2) + \\
+ \frac{dx_2}{x_2} (a_{12}x_1x_2 -a_{23}x_3x_2 
-a_{212}x_2^2+a_{121}x_1^2+a_{323}x_3^2 +\\
+a_{123}x_1x_3-a_{232}x_2^2-a_{312}x_3x_2) +\\
+\frac{dx_3}{x_3}(a_{13}x_1x_3 
+a_{23}x_2x_3-a_{313}x_3^2-a_{323}x_3^2+a_{131}x_1^2+\\
+a_{232}x_2^2-a_{123}x_1x_3+a_{231}x_1x_2).
\end{multline*}

\end{cor}
\begin{proof}
This is an easy verification.
\end{proof}

\begin{prop}\label{dimensions}

1) Assume that $S$ is a  nodal Burniat surface with  $K_S^2 = 4$ ($m=2$).
Then the dimension of the  vector space
$$ H^0 ( \Omega_{\tilde{Y}}^1 (\log (D_i)) (K_{\tilde{Y}} + L_i  ))
     = H^0 ( \Omega_{\tilde{Y}}^1 (\log (D_i)) (E_i - E_{i+2}) )$$
is $1$ for $i=3$, else it is $0$.

2) Consider instead extended Burniat divisors for $m=2$, and the
corresponding vector spaces
$$ H^0 ( \Omega_{\tilde{Y}}^1 (\log (\De_i)) (K_{\tilde{Y}} + \Lambda_i  )).$$
Then their dimensions are the same as in the Burniat case, namely,
$1$ for $i=3$, else $0$.

3) Assume that $S$ is a  Burniat surface with  $K_S^2 = 3$ ($m$=3).

Then each vector space
$$ H^0 ( \Omega_{\tilde{Y}}^1 (\log (D_i)) (K_{\tilde{Y}} + L_i  ))
     = H^0 ( \Omega_{\tilde{Y}}^1 (\log (D_i)) (E_i - E_{i+2}) )$$
is equal to 0.

4) In the case of (strictly or not strictly) extended  Burniat
divisors for $m=3$ we have $\forall i$:
$$ H^0 ( \Omega_{\tilde{Y}}^1 (\log (\De_i)) (K_{\tilde{Y}} + \Lambda_i
))  = 0.$$
\end{prop}

\begin{proof}

We can  prove 1) and 2) simultaneously for $i=1$.

Observe that $D_1 = \De_1 + N_1$, that
$ \Lambda_1 = L_1 + N_1$, and apply  Lemma \ref{poles} (observing 
that $(K_{\tilde{Y}} + 2N_1 + (E_1-E_3))N_1  = -4 +1 <0$)
in order to conclude that
    $$H^0 ( \Omega_{\tilde{Y}}^1 (\log (\De_1))
(E_1 - E_{3} + N_1 )) \cong H^0 ( \Omega_{\tilde{Y}}^1 (\log (D_1))(E_1 - E_{3} )).$$

Moreover we observe that, again by lemma \ref{poles},
$$H^0 ( \Omega_{\tilde{Y}}^1 (\log (D_1))(E_1 - E_{3} )) = H^0 ( \Omega_{\tilde{Y}}^1 (\log (D_1-(L-E_1)))(L- E_{3} )).$$
Let $f \colon \tilde{Y} \ra \PP^2$ be the blow down of  $E_1,
\ldots , E_5$. Then $f_*(D_1 -(L-E_1))$ splits as the sum of two
lines
$l_1,l_2$ in $\PP^2$ intersecting in $P_1$.

W.l.o.g. we can assume
that $P_1 = (0:0:1)$, $P_2=(0:1:0)$, $P_4=(1:0:0)$ and $P_5 =
(1:0:\lambda)$, with $\lambda \neq 0$.

Applying  proposition
\ref{Hauptlemma} several times for each blow down we get that
$$H^0 ( \Omega_{\tilde{Y}}^1 (\log (D_1-(L-E_1)))(L- E_{3} )) = 
H^0 ( f_*\Omega_{\tilde{Y}}^1 (\log (D_1-(L-E_1)))(L- E_{3} ))$$
is the subspace $V_1$ of
$H^0 ( \Omega_{\PP^2}^1 (\dlog l_1,\dlog l_2) (1))$ consisting of 
sections satisfying several  linear conditions.

We write these conditions using the basis provided by lemma 
\ref{cross} and its corollary, in order to show
that $V_1 =0$.
By prop. \ref{Hauptlemma}, 3) we get for $P_1$:
$$
a_{13} + a_{23} = 0;
$$
for $P_2$, $P_4$ and $P_5$ the three equations
$$
a_{12} = a_{21} =a_{21}+ \lambda a_{23} =0.
$$
This shows that $V_1 = 0$.

We continue with the proof of 1).

For $i=2$, again by lemma \ref{poles} we have to calculate
$$
V_2 = H^0(\Omega_{\tilde{Y}}^1 (\log (L-E_2-E_5), \log(L-E_2-E_4)) (L - E_3 )),
$$
which after blowing down $E_1, \ldots ,E_5$ corresponds to a subspace 
of $H^0 ( \Omega_{\PP^2}^1 (\dlog l_1,\dlog l_2) (1))$.

W.l.o.g. we can assume that $P_2 = (0:0:1)$, $P_5=(0:1:0)$, 
$P_4=(1:0:0)$ and $P_3 =
(1:1:1)$.

By prop. \ref{Hauptlemma}, 3), we get for $P_2$,$P_4$, $P_5$ the 
three linear equations:
$$
a_{13}+a_{23} = 0, \ a_{21} = 0, \ a_{12} = 0.
$$
We evaluate $\omega$ in $P_3$, and get (using  the above equalities)
$$
\omega(P_3) = a_{13}dx_1 + a_{23}dx_2,
$$
whence by proposition
\ref{Hauptlemma}, 6) $a_{13}=a_{23} =0$ and therefore we have 
verified that $V_2 =0$.

For $i=3$, using lemma \ref{poles}, we have to calculate:
$$
V_3:= H^0(\Omega_{\tilde{Y}}^1 (\log (L-E_3-E_4), \log(L-E_3-E_5)) (L - E_1 )),
$$
which, after blowing down $E_1, \ldots, E_5$, becomes a linear 
subspace of $H^0 ( \Omega_{\PP^2}^1 (\dlog l_1,\dlog l_2) (1))$.

W.l.o.g. we can assume that $P_3 = (0:0:1)$, $P_4=(0:1:0)$, 
$P_5=(1:0:0)$,   $P_1 =
(1: 1 :0)$.

By prop. \ref{Hauptlemma}, 3), we get for $P_3$,$P_4$, $P_5$ the 
three linear equations:
$$
a_{13}+a_{23} = 0, \ a_{12} = 0, \ a_{21} = 0.
$$
Setting the evaluation of $\omega$ in $P_1$ equal to zero is easily 
seen  to give no new conditions, hence $V_3 \cong \CC$.

Let's proceed to prove  2) for $i=2,3$.

     For $i=2,3$, by 4) of remark \ref{difference},  $$ H^0 ( 
\Omega_{\tilde{Y}}^1 (\log \De_i) (K_{\tilde{Y}} +
\Lambda_i  )) = H^0 ( \Omega_{\tilde{Y}}^1 (\log \De_i) (E_i - E_{i+2} )).$$

For $i=2$, using again lemma \ref{poles}, observing that
$$
(K_{\tilde{Y}} + 2\Gamma_2 + (E_2-E_1)) \Gamma_2 <0,
$$
we see that we have to calculate
$$
V_2:= H^0(\Omega_{\tilde{Y}}^1 (\log (L-E_2-E_4), \log(L-E_2-E_5)) 
(2L - E_1-E_3-E_4-E_5 )),
$$
which, after blowing down $E_1, \ldots, E_5$, becomes a linear 
subspace of $H^0 ( \Omega_{\PP^2}^1 (\dlog l_1,\dlog l_2) (2))$.

W.l.o.g. we can assume that $P_2 = (0:0:1)$, $P_4=(0:1:0)$, 
$P_5=(1:0:0)$, $P_3 =(1:1:1)$, and then $P_1 =
(1:\lambda:0)$, where $\lambda \neq 0,1$.

Using cor. \ref{dxi}, we get by prop. \ref{Hauptlemma}, 3) for $P_2$ 
the linear equation
$$
a_{313}+a_{323} = 0.
$$
By prop. \ref{Hauptlemma}, 5) the conditions for $P_4$ are
$$
a_{212}=0, \ a_{12} = 0, \ a_{23} = 0;
$$

whereas the conditions for  $P_5$ are
$$
a_{121}=0, \ a_{12} = 0, \ a_{13} = 0.
$$
Imposing that  $\omega$ vanishes  in $P_3$, we get
$$
\omega(P_3) = dx_1(a_{313} +2a_{213}+a_{123}) +dx_2(a_{323} 
+2a_{123}+a_{213})= 0.
$$

The above conditions  yield:
$$
a_{123}=a_{313} = - a_{213} = -a_{323}.
$$

Finally, imposing that  $\omega$ vanishes   in $P_1$ we obtain:
$$
\omega(P_1) = -dx_3 (\lambda a_{213} + a_{123})=0,
$$
whence $(\lambda -1) a_{213} = 0$. Since $\lambda \neq 0,1$ this implies
$a_{213} =0$, and we have shown that $V_2 =0$.

We are left with the case $i=3$. Using repeatedly lemma \ref{poles} 
and proposition \ref{Hauptlemma}, we see that we have to calculate
\begin{multline*}
V_3 :=H^0 ( \Omega_{\tilde{Y}}^1 (\log \De_3) (E_3 - E_2 )) = \\ = 
H^0(\Omega_{\tilde{Y}}^1 (\log (L-E_1-E_3), \log(L-E_1-E_4-E_5)) (2L 
- E_3-E_4-E_5 )).
\end{multline*}
After blowing down $E_1, \ldots, E_5$, we can assume w.l.o.g.
  that $P_1 = (0:0:1)$, $P_4=(0:1:0)$, $P_3=(1:0:0)$,   $P_5 =
(0:1: 1)$, and $V_3$ becomes a linear subspace of $H^0 ( 
\Omega_{\PP^2}^1 (\dlog l_1,\dlog l_2) (2))$.

Using cor. \ref{dxi}, we get by prop. \ref{Hauptlemma}, 3) for $P_1$ 
the linear equation
$$
a_{313}+a_{323} = 0.
$$
By prop. \ref{Hauptlemma}, 5) the conditions for $P_4$ are
$$
a_{212}=0, \ a_{12} = 0, \ a_{23} = 0;
$$

whereas the conditions for  $P_3$ are
$$
a_{121}=0, \ a_{12} = 0, \ a_{13} = 0.
$$
For  $P_5$, we get instead (again by prop. \ref{Hauptlemma}, 5)), the 
two linear equations (the third is trivial):
$$
a_{213} +  a_{313} = 0, \ 2  a_{313} = 0.
$$
This implies that $a_{313} =  a_{213} = a_{323} =0$, but $a_{123}$ is 
arbitrary. This shows that $V_3 \cong \CC$.

Thus 2) is proven.

To prove 3), by symmetry, we may assume without loss of generality that $i=1$.

We have to calculate  $V_1:=H^0( \Omega_{\tilde{Y}}^1 (\log (D_1) (E_1
- E_{3} )) $, which by lemma \ref{poles} is equal to 
$$H^0(\Omega_{\tilde{Y}}^1 (\log (L-E_1-E_2), \log(L-E_1-E_4-E_5)) (L 
- E_6 )),
$$
which, after blowing down $E_1, \ldots, E_5$, becomes a linear 
subspace of $H^0 ( \Omega_{\PP^2}^1 (\dlog l_1,\dlog l_2) (1))$.

W.l.o.g. we can assume that $P_1 = (0:0:1)$, $P_5=(0:1:0)$, 
$P_2=(1:0:0)$,  $P_4 =
(0:1: 1)$. Since $P_2, P_4,P_6$ are collinear, $P_6 = (1: \mu: \mu 
)$, where $\mu \neq
0$.

By prop. \ref{Hauptlemma}, 3), we get for $P_1$,$P_2$, $P_4$ and 
$P_5$ the  linear equations:
$$
a_{13}+a_{23} = 0, \ a_{21} = 0, \ a_{12} +  a_{13} = 0, \ a_{12} = 0.
$$
This already shows that $V_1 =0$.

Thus 3) is proven.

Let us treat the subcase of 4) where we have strictly extended 
Burniat divisors:
the situation is here symmetric in the indices $i$, hence it suffices
to show the vanishing of
$$ H^0 ( \Omega_{\tilde{Y}}^1 (\log (\De_1)) (K_{\tilde{Y}} + \Lambda_1
)).$$
Recall that we have the  decomposition in irreducible
connected components $\De_1 =  G_1 + \Ga_1 + N_2$,
where $G_1 $ is the  del Pezzo line $ G_1 \equiv L - E_1 - E_6$.

By lemma \ref{poles} we get:
$$ H^0 ( \Omega_{\tilde{Y}}^1 (\log (\De_1)) (K_{\tilde{Y}} + \Lambda_1
))  =  H^0 ( \Omega_{\tilde{Y}}^1 (\log (\De_1 + N_1)) (E_1-E_3)),$$

since $(K_{\tilde{Y}} +2N_1 +(E_1-E_3))N_1 < 0$.
Using again lemma \ref{poles} we see that
\begin{multline*}
  H^0 ( \Omega_{\tilde{Y}}^1 (\log (\De_1 + N_1)) (E_1-E_3))=\\
=H^0 ( \Omega_{\tilde{Y}}^1 (\log (\De_1 + N_1 - \Gamma_1)) 
((E_1-E_3)+ \Gamma_1))=\\
= H^0(\Omega_{\tilde{Y}}^1 (\log(G_1+N_1+N_2))(2L-E_2-E_3-E_5-E_6)),
\end{multline*}

because $(K_{\tilde{Y}} +2\Gamma_1 +(E_1-E_3))\Gamma_1 < 0$.

Let $f \colon \tilde{Y} \ra \PP^2$ be the blow down of  $E_1,
\ldots , E_6$. Then $f_*(G_1 + N_1 + N_2)$ splits as the sum of three
lines
$l_1,l_2,l_3$ in $\PP^2$ forming a triangle. W.l.o.g. we can assume
that $P_6 = (1:0:0)$, $P_1=(0:1:0)$, $P_4=(0:0:1)$ and $P_3 =
(1:1:1)$. Then $P_5 = (0:1:1)$, whereas $P_2$ is collinear with
$P_6$, $P_4$, whence $P_2 = (1:0:\lambda)$, with $\lambda \neq 0,1$.

Then

\begin{multline*}
  H^0 ( \Omega_{\tilde{Y}}^1 (\log(G_1+N_1+N_2))(2L-E_2-E_3-E_5-E_6)) = \\
H^0(f_*\Omega_{\tilde{Y}}^1 (\log(G_1+N_1+N_2))(2L-E_2-E_3-E_5-E_6))
\end{multline*}
is a subspace of
$$H^0 ( \Omega_{\PP^2}^1 (\dlog l_1,\dlog l_2, \dlog l_3 ) (2)),$$
  where $l_i = x_i$, whence $P_1, P_4,P_5 \in \{l_1=0\}$, $P_6, P_4, P_2 
\in \{l_2 =0\}$, $P_1, P_6 \in \{l_3 =0\}$,
consisting of sections satisfying fourteen linear conditions 
described in proposition
\ref{Hauptlemma}.

We explicitly write these conditions using lemma \ref{triangle} and 
lemma \ref{dxi} in order to show
that this subspace must be trivial.

Let $\omega \in H^0 ( \Omega_{\PP^2}^1 (\dlog l_1,\dlog l_2, \dlog 
l_3 ) (2))$ and we write $\omega$
  in the basis of lemma \ref{triangle}:
\begin{multline*}
  \omega = a_{12} \eta_{12} +a_{13} \eta_{13} +a_{23} \eta_{23}
+a_{212}x_2 \omega_{12}+a_{313}x_3 \omega_{13}+
a_{323}x_3 \omega_{23}+\\
+a_{121}x_1 \omega_{21}
+a_{131}x_1 \omega_{31}+a_{232}x_2 \omega_{32}+a_{123}x_1 
\omega_{23}+a_{231}x_2 \omega_{31}+a_{312}x_3 \omega_{12}.
\end{multline*}
Then by prop. \ref{Hauptlemma}, 3) the condition for $P_1 =(0:1:0)$  is
\begin{equation}
  a_{212}+a_{232} = 0.
\end{equation}
The same argument shows that the linear condition for $P_4 = (0:0:1)$ is
\begin{equation}
  a_{313}+a_{323} = 0.
\end{equation}
Next we work out the conditions for $P_5,P_2$ using prop. 
\ref{Hauptlemma}, 5). For $P_5 :=(0:1:1)$
we work in the chart $x_3 = 1$ and write $\omega$
  locally around $(0,1)$ as $\alpha(x_1,x_2) dx_2 + \beta (x_1,x_2) 
\frac{dx_1}{x_1}$. Then we get (using lemma
\ref{dxi}):

\begin{equation}
  \beta(0,1) = a_{212}+a_{313}+a_{312} = 0;
\end{equation}

\begin{equation}
\frac{\partial{\beta}}{\partial{x_1}} (0,1)= -a_{12}-a_{13} -a_{231} = 0;
\end{equation}
\begin{equation}
  \frac{\partial{\beta}}{\partial{x_2}} (0,1) + \alpha(0,1) = 
-a_{23}+2a_{212}+a_{323} = 0.
\end{equation}
The same argument for $P_2 = (1:0:\lambda)$ (working in the chart 
$x_1 = 1$ and writing $\omega$
  locally around $(0,\lambda)$ as $\alpha(x_2,x_3) dx_3 + \beta 
(x_2,x_3) \frac{dx_2}{x_2}$) gives the following three
linear equations ($\lambda \neq 0,1$):
\begin{equation}
  \beta(0, \lambda) = a_{323} \lambda^2+a_{121}+a_{123}\lambda = 0;
\end{equation}
\begin{equation}
\frac{\partial{\beta}}{\partial{x_2}} (0,\lambda) =  a_{12}-\lambda 
a_{23} - \lambda a_{312}   = 0;
\end{equation}
\begin{multline}
  \frac{\partial{\beta}}{\partial{x_3}} (0,\lambda) + \alpha(0,\lambda) =
  a_{13}+\frac{1}{\lambda}a_{131}+\lambda a_{323} - \lambda a_{313} =\\
= a_{13}+\frac{1}{\lambda}a_{131}+2\lambda a_{323}= 0.
\end{multline}

There are four linear conditions coming from $P_6=(1:0:0)$, given in 
prop. \ref{Hauptlemma}, 7).
We work in the chart $x_1 = 1$ and write $\omega = \alpha (x_2,x_3) 
\frac{dx_2}{x_2} + \beta(x_2,x_3)
\frac{dx_3}{x_3}$. Then we get:
\begin{equation}
  \alpha(0,0)=a_{121} = 0;
\end{equation}
\begin{equation}
  \beta(0,0)=a_{131} = 0;
\end{equation}
\begin{equation}
\frac{\partial{(\alpha + \beta)}}{\partial{x_2}} (0,0) =  a_{12}+a_{231} = 0;
\end{equation}
\begin{equation}
\frac{\partial{(\alpha + \beta)}}{\partial{x_3}} (0,0)= a_{13} = 0.
\end{equation}

 From equation (11) we get: $a_{12} = -a_{231}$.

Since $a_{13}=a_{131}=0$, equation (8) implies $a_{323}=0$, whence by 
(2) also $a_{313}=0$.
Moreover, by (6), we get $a_{123} =0$.

We write finally the conditions coming from $P_3 = (1:1:1)$
(using again that certain coefficients are zero).

We evaluate
$\omega$ in $P_3$ and work in the affine chart $x_2 = 1$ to obtain
\begin{multline}
\omega(P_3) = (-a_{12}+ a_{212}-a_{231}+a_{312})dx_1 +\\
(a_{23}+a_{232}+a_{231})dx_3 = 0.
\end{multline}
Since:
\begin{equation*}
  \omega(P_3) = ( a_{212} +a_{312})dx_1 + ((a_{23}+a_{232}+a_{231})dx_3
\end{equation*}

we get the last two linear equations:
\begin{equation}
  a_{231}+a_{23}+a_{232} = 0;
\end{equation}

\begin{equation}
  a_{212} +a_{312} = 0.
\end{equation}

These immediately imply that
$$
a_{312}=a_{232} = -a_{212}.
$$
By (14) $a_{23} = -a_{232}-a_{231}$, and using (5), we see that 
$a_{23} = 2a_{212}$.

Again by (14) we get then that $a_{212} = -a_{231}$.
Therefore, we have:
$$
a_{212} = -a_{231} = -a_{312}=-a_{232} = \frac{a_{23}}{2}.
$$

By (7):
$$
0=a_{12}- \lambda a_{23} - \lambda a_{231} = a_{12} + \lambda a_{231},
$$

whence by (4) $\lambda = 1$, which gives a contradiction, or 
$a_{12}=a_{231}=0$.

Hence the claim for strictly extended Burniat surfaces with $K_S^2 =3$
is established.

Next we come to the case of (non strictly) extended Burniat surfaces.
Here we have to consider two cases:
\begin{itemize}
   \item[a)] only one of the three conics $\Gamma_i$ degenerates to two lines;
\item[b)] exactly two of the three conics $\Gamma_i$ degenerate to two lines.
\end{itemize}

a) W.l.o.g.  and by remark \ref{difference} 5-3 we may assume that 
$\Gamma_1$ splits as
$$
\Gamma_1 \equiv (L-E_1-E_2) +N_3+E_3.
$$

Then we get the extended Burniat divisors:
$$
        \{D_1'\} = |L-E_1 - E_6|+|L - E_1 -E_2| + E_3 + N_2,
        $$
        $$
        D_2' \in |L-E_2 - E_5|+|2L - E_2 -E_3-E_4-E_5|,
        $$
        $$
        D_3' \in |L-E_3 - E_4|+|2L - E_1 -E_3-E_4-E_6| + N_1+N_3.
        $$

We make  the assumption, for each $D'_i$, $i=2,3$
        that the strict transform of the  conic  $\Ga_i$ is irreducible.

Then we have
$$ K_{\tilde{Y}} + \mathcal{L}_1' \equiv L-E_3-E_4-E_5 \equiv 
K_{\tilde{Y}} + \Lambda_1,
$$
$$ K_{\tilde{Y}} + \mathcal{L}_2' \equiv L-E_1-E_4-E_6 \equiv
K_{\tilde{Y}} + \Lambda_2,
$$
$$ K_{\tilde{Y}} + \mathcal{L}_3' \equiv E_3-E_2 \equiv K_{\tilde{Y}}
+ \Lambda_3 -N_3,
$$

where the $\Lambda_i$ are as for the  strictly extended Burniat
divisors.

Observe that $D_2' + N_3 = \Delta_2$, whence
\begin{multline*} H^0(\Omega^1_{\tilde{Y}}(\log D_2')(K_{\tilde{Y}} +
\mathcal{L}_2')) \subset  H^0(\Omega^1_{\tilde{Y}}(\log
D_2')(K_{\tilde{Y}} + \mathcal{L}_2' + N_3))= \\
H^0(\Omega^1_{\tilde{Y}}(\log (D_2'+N_3))(K_{\tilde{Y}} +
\mathcal{L}_2'))=H^0(\Omega^1_{\tilde{Y}}(\log \Delta_2)(K_{\tilde{Y}}
+ \Lambda_2)) = 0,
\end{multline*} where the first equality holds by lemma \ref{poles}
and the last holds by our previous computations for strictly
extended Burniat surfaces.

Moreover,  $D_3'  = \Delta_3 +N_3$, $\mathcal{L}_3' = \Lambda_3-N_3$,
whence the vanishing of $H^0(\Omega^1_{\tilde{Y}}(\log
D_3')(K_{\tilde{Y}} + \mathcal{L}_3'))$ follows again using lemma
\ref{poles} from the analogous vanishing for strictly extended
Burniat surfaces. It remains to prove the following
\begin{claim}
   $H^0(\Omega^1_{\tilde{Y}}(\log D_1')(K_{\tilde{Y}} + \mathcal{L}_1')) = 0$
\end{claim}

\begin{proof}[Proof of the claim.]
   By lemma \ref{poles} we see that
$$ H^0(\Omega^1_{\tilde{Y}}(\log D_1')(K_{\tilde{Y}} + \mathcal{L}_1'))
= H^0(\Omega^1_{\tilde{Y}}(\log (D_1'-E_3))(L - E_4 - E_5)).
$$ Let $f \colon \tilde{Y} \ra \PP^2$ be the blow down of  $E_1,
\ldots , E_6$.

Then $f_*(D_1'-E_3)$ splits as the sum of three
lines
$l_1,l_2,l_3$ in $\PP^2$ forming a triangle. W.l.o.g. we can assume
that $P_1 = (1:0:0)$, $P_2=(0:1:0)$, $P_6=(0:0:1)$ and $P_5 =
(1:1:1)$. Then $P_4 = (0:1:1)$.

We conclude that $H^0(\Omega^1_{\tilde{Y}}(\log (D_1'-E_3))(L - E_4 - E_5)) =
H^0(f_*\Omega^1_{\tilde{Y}}(\log (D_1'-E_3))(L - E_4 - E_5))$
is the subspace of $H^0 ( \Omega_{\PP^2}^1 (\dlog l_1,\dlog l_2, 
\dlog l_3 ) (1))$
consisting of sections satisfying one linear
condition for
$P_1,P_2,P_6$ each, two linear conditions for $P_5$ and three linear
conditions for $P_4$, described in proposition
\ref{Hauptlemma}.

We write these conditions using lemma \ref{triangle} in order to show
that this subspace must be trivial.

By lemma \ref{triangle} we write $\omega \in H^0 ( \Omega_{\PP^2}^1
(\dlog l_1,\dlog l_2, \dlog l_3 ) (1)$ as
$$
\omega = \sum_{i\neq j} a_{ij}\omega_{ij}.
$$ Then the three equations for $P_1,P_2,P_6$ (cf. prop.
\ref{Hauptlemma}, 3)) are
$$ a_{21} +a_{31} =0, \ a_{12} +a_{32} =0, \ a_{13} +a_{23} =0.
$$

By prop. \ref{Hauptlemma}, 5), we get for $P_4$ the linear equations:
$$
a_{12}+a_{13}=0, \ -a_{21}-a_{31}=0, \ a_{23}-a_{32}=0.
$$

The above conditions already imply:
$$
a_{13}=a_{12}=a_{23}=a_{32}=0, \ a_{21}= - a_{31}.
$$
We impose the vanishing of $\omega$ in $P_5 = (1:1:1)$ working in the 
affine chart
$x_3 = 1$ and obtain
$$
\omega(1:1:1) = (-a_{21} - a_{31})dx_1 + (a_{21} )dx_2= 0,
$$
whence $a_{21}=a_{31}=0$.
\end{proof}

   b) W.l.o.g. we can assume that each of the two conics $\Gamma_1$ and
$\Gamma_2$ degenerate to two lines. Then
we get the extended Burniat divisors:
$$
        \{D_1''\} = |L-E_1 - E_6|+|L - E_1 -E_2| + E_3 + N_1 + N_2,
        $$
        $$
        \{D_2''\} = |L-E_2 - E_5|+|L - E_2 -E_3| + E_1,
        $$
        $$
        D_3'' \in |L-E_3 - E_4|+|2L - E_1 -E_3-E_4-E_6| + N_3.
        $$
We make  the assumption for $D_3''$
        that the strict transform of the  conic $\Ga_3$ passing through
$P_1,P_3,P_4,P_6$ is irreducible.

Then we have
$$ K_{\tilde{Y}} + \mathcal{L}_1'' \equiv E_1-E_3 = K_{\tilde{Y}} +
\mathcal{L}_1' - N_1, \ D_1'' = D_1'+N_1;
$$
$$ K_{\tilde{Y}} + \mathcal{L}_2'' \equiv L-E_1-E_4-E_6 \equiv
K_{\tilde{Y}} + \mathcal{L}_2 +N_2, \ D_2'' = D_2 -N_2;
$$
$$ K_{\tilde{Y}} + \mathcal{L}_3'' \equiv E_3-E_2 \equiv
K_{\tilde{Y}} + \Lambda_3', \ D_3'' = D_3'-N_3.
$$

Therefore for $i=1,2,3$ the vanishing of $H^0(\Omega^1_{\tilde{Y}}(\log
D_i'')(K_{\tilde{Y}} + \mathcal{L}_i''))$ can be reduced
via  lemma \ref{poles} to the analogous vanishing for extended
Burniat surfaces of case a) ($i=1,3$) and to the analogous
vanishing for Burniat divisors (for $i=2$), which was already proved
in part 3).

\end{proof}

Now that the proof of proposition \ref{dimensions} is finally accomplished,
  we can explicitly determine the several character spaces for $H^i(S, 
\Theta_S)$ and their
dimensions.

In the following, given a $G$-space $V$, $ V^i$, for $ i \in 
{1,2,3}$, denotes the eigenspace
corresponding to  the character
whose kernel consists of $\{ 1, g_i\}$.

\begin{prop}\label{locmod} 1) Let $S$ be the minimal model of a
Burniat surface.

Then the dimensions of  the  eigenspaces  of the cohomology groups of
the tangent sheaf $
\Theta_S$ (for the natural $(\ZZ/ 2\ZZ)^2$-action) are as follows.

\begin{enumerate}
          \item
$\underline{K^2 = 4}$ of nodal type:  \subitem $h^1(S,
\Theta_S)^{\inv} = 2$, $h^2(S,
\Theta_S)^{\inv} = 0$,

\subitem $h^1(S, \Theta_S)^3 = 1 = h^2(S, 
\Theta_S)^3$,
\subitem
         $h^j(S, \Theta_S)^i  = 0$, for $i \in 
\{1,2 \}$;

\item
$\underline{K^2 = 3}$:
\subitem $h^1(S, 
\Theta_S)^{\inv} = 1$, $h^2(S,
\Theta_S)^{\inv} =0$,
\subitem
$h^1(S, 
\Theta_S)^i = 1$, $h^2(S, \Theta_S)^i= 0$, for $i \in \{1,2,3 
\}$.
\end{enumerate}

\noindent 2) Let $S$ be a minimal model of an 
extended   Burniat surface with $K_S^2 =4$.

Then the dimensions of 
the  eigenspaces  of the cohomology groups of 
the tangent sheaf 
$
\Theta_S$ (for the natural $(\ZZ/ 2\ZZ)^2$-action) are as 
follows.

\begin{itemize}
    \item $h^1(S,
\Theta_S)^{\inv} = 3$, 
$h^2(S,
\Theta_S)^{\inv} = 0$,

\item $h^1(S, \Theta_S)^i = 0 = 
h^2(S, \Theta_S)^i$, for $i \in
\{1,2 \}$,
\item $h^1(S, \Theta_S)^3 
= 1 = h^2(S, \Theta_S)^3$.
\end{itemize}
\noindent 3) Let $S$ be the 
minimal model of an extended 
Burniat surface with $K_S^2 =3$.

Then 
the dimensions of  the  eigenspaces  of the cohomology groups of 
the 
tangent sheaf $
\Theta_S$ (for the natural $(\ZZ/ 2\ZZ)^2$-action) 
are as follows:

\begin{enumerate}
         \item strictly extended: 
\subitem $h^1(S,
\Theta_S)^{\inv} = 4$, $h^2(S,
\Theta_S)^{\inv} = 
0$,
\subitem
         $h^j(S, \Theta_S)^i  = 0$, for $i \in \{1,2,3 
\}$;

\item the conic $\Gamma_1$ degenerates to two lines:
\subitem 
$h^1(S, \Theta_S)^{\inv} = 3$, $h^2(S,
\Theta_S)^{\inv} 
=0$,
\subitem
$h^1(S, \Theta_S)^i = 0 =h^2(S, \Theta_S)^i$, for $i 
\in \{1,3 \}$,
\subitem $h^1(S, \Theta_S)^2 = 1$, $ h^2(S, 
\Theta_S)^2=0$;
\item the conics $\Gamma_1$, $\Gamma_2$ degenerate to 
two lines each:
\subitem $h^1(S, \Theta_S)^{\inv} = 2$, 
$h^2(S,
\Theta_S)^{\inv} =0$,
\subitem
$h^1(S, \Theta_S)^1 = 0 
=h^2(S, \Theta_S)^1$, 
\subitem $h^1(S, \Theta_S)^i = 1$, $h^2(S, 
\Theta_S)^i = 0$, for $i \in \{1,3 
\}$.

\end{enumerate}

\end{prop}

\begin{proof}
For the invariant 
part, the calculation goes exactly as the proof of lemma 2.9. of 
\cite{burniat2}, using that $h^i(\Theta_{\tilde{S}})^{inv} = 
h^i(\Theta_S)^{inv}$.

For the other character spaces, we use the 
same argument as in lemma 2.12. in \cite{burniat2} to calculate 
$\chi(\Omega^1_{\tilde{Y}}(\log D_i)(K_{\tilde{Y}} + \mathcal{L}_i))$ 
(resp. $\chi(\Omega^1_{\tilde{Y}}(\log \Delta_i)(K_{\tilde{Y}} + 
\Lambda_i))$ for extended Burniat surfaces).

We first observe that 

$$\chi(\Omega^1_{\tilde{Y}}(\log D_i)(K_{\tilde{Y}} + 
\mathcal{L}_i)) = \chi(\Omega^1_{\tilde{Y}}(K_{\tilde{Y}} + 
\mathcal{L}_i)) + \chi(\hol_{D_i}(\log D_i)(K_{\tilde{Y}} + 
\mathcal{L}_i)),$$ (and analogously for 
$\chi(\Omega^1_{\tilde{Y}}(\log \Delta_i)(K_{\tilde{Y}} + 
\Lambda_i))$ for extended Burniat surfaces).

Moreover, note that 
with the same calculation as in lemma 2.12. of \cite{burniat2}, we 
see that $\chi(\Omega^1_{\tilde{Y}}(K_{\tilde{Y}} + \mathcal{L}_i)) = 
K_{\tilde{Y}}^2 -12$.

Each $D_i$ (resp. $\Delta_i$) consists of 
$k_i$ irreducible connected components, each of them being a smooth 
rational curve. Write $D_i = D_{i,1} + \ldots + D_{i,k_i}$ as 
disjoint union of smooth rational curves and let $n_j:= D_{i,j} \cdot 
(K_{\tilde{Y}} + L_i)$. Then 
$$
\chi(\hol_{D_i}(\log 
D_i)(K_{\tilde{Y}} + \mathcal{L}_i)) = \sum_{j=1}^{k_j} 
max(0,n_j+1).
$$
Therefore 
$$\chi(\Omega^1_{\tilde{Y}}(\log 
D_i)(K_{\tilde{Y}} + \mathcal{L}_i)) = K_{\tilde{Y}}^2 -12 
+\sum_{j=1}^{k_j} max(0,n_j+1).
$$

We summarize the calculations in 
the following table (note that we write $\chi_i$ for 
$\chi(\Omega^1_{\tilde{Y}}(\log D_i)(K_{\tilde{Y}} + 
\mathcal{L}_i))$). The values for $h^2(\Theta_{\tilde{S}})^i$ have 
been calculated in prop. \ref{dimensions}.
The notation:  extended 
case (2), resp. (3), refers to proposition \ref{locmod}.

 Moreover, 
we use lemma 9.22 of 
\cite{cime88} to compare 
$h^1(\Theta_{\tilde{S}})$ and $h^1(\Theta_S)$: it asserts that for a 
single blow up of a point $P$
$$ \pi_* \Theta_{\tilde{S}} = 
\mathfrak M_P \Theta_S, \ \ \mathcal R ^1\pi_* \Theta_{\tilde{S}} = 
0.$$

\begin{small}
\begin{table}[ht]
\label{chi}
\begin{tabular}{|c|c|c|c|c|c|c|c|}
\hline
$K_S^2$&$i$&$(n_1, 
\ldots, n_{k_i})$& $\chi_i$  & 
$h^2(\Theta_{\tilde{S}})^i$&$h^1(\Theta_{\tilde{S}})^i$&$h^2(\Theta_S)^i$&$h^1(\Theta_S)^i$ 
\\
\hline\hline
&$0$&&$-2$&$0$&$2$&$0$&$2$\\
$4$ &$1$& 
$(1,1,1,1)$&$0$&$0$&$0$&$0$&$0$\\
n.B. &$2$& 
$(1,1,1,1)$&$0$&$0$&$0$&$0$&$0$\\
 &$3$& 
$(1,1,1,1)$&$0$&$1$&$1$&$1$&$1$\\
\hline
&$0$&&$-3$&$0$&$3$&$0$&$3$\\
$4$ 
&$1$& $(1,1,1)$&$-2$&$0$&$2$&$0$&$0$\\
ext. &$2$& 
$(1,1,1)$&$-2$&$0$&$2$&$0$&$0$\\
 &$3$& 
$(1,1,1,0,1)$&$1$&$1$&$0$&$1$&$0$\\
\hline
&$0$&&$-1$&$0$&$1$&$0$&$1$\\
$3$ 
&$1$& $(1,1,1,1)$&$-1$&$0$&$1$&$0$&$1$\\
B. &$2$& 
$(1,1,1,1)$&$-1$&$0$&$1$&$0$&$1$\\
 &$3$& 
$(1,1,1,1)$&$-1$&$0$&$1$&$0$&$1$\\
\hline
\hline
&$0$&&$-4$&$0$&$4$&$0$&$4$\\
$3$ 
&$1,2,3$& $(1,1,0)$&$-4$&$0$&$4$&$0$&$0$\\
str.ext. 
&&&&&&&\\
\hline
\hline
&$0$&&$-3$&$0$&$3$&$0$&$3$\\
$3$ &$1$& 
$(1,1,0,1)$&$-2$&$0$&$2$&$0$&$0$\\
ext. (2) &$2$& 
$(1,1)$&$-5$&$0$&$5$&$0$&$1$\\
 &$3$& 
$(1,1,0,1)$&$-2$&$0$&$2$&$0$&$0$\\
\hline
\hline
&$0$&&$-2$&$0$&$2$&$0$&$2$\\
$3$ 
&$1$& $(1,1,1,1,0)$&$0$&$0$&$0$&$0$&$0$\\
ext. (3). &$2$& 
$(1,1,1)$&$-3$&$0$&$3$&$0$&$1$\\
 &$3$& 
$(1,1,1)$&$-3$&$0$&$3$&$0$&$1$\\
\hline

\hline
\end{tabular}
\end{table}
\end{small}

\end{proof}

 From the  above calculations and from propositions \ref{famiglia4}, 
\ref{famiglia3} follow
all the statements of our   first main 
theorem, with the exception of the statement 
that $\sN \sE \sB_4$ is 
a connected component. It follows that $\sN \sE \sB_4$ is open,
while 
the statement that  $\sN \sE \sB_4$ is closed will be shown in the 
forthcoming section.

{ \bf Theorem \ref{main1}}
 {\em      1) The 
subset  $\sN \sE \sB_4$ of the  moduli
space of canonical
surfaces of 
general type $\mathfrak M^{can}_{1,4}$   given by  the union of the 
open set corresponding to
{\em extended } Burniat surfaces with 
$K^2_S = 4$ with the irreducible closed set parametrizing nodal 
Burniat surfaces with $K^2_S = 4$ 
is an irreducible connected 
component, normal, unirational
        of  dimension 3.

Moreover  the base of the Kuranishi family of deformations of such a 
minimal model $S$ is smooth.

\noindent
      2) The subset  $\sN \sE 
\sB_3$ of the  moduli
space of canonical
surfaces of general type 
$\mathfrak M^{can}_{1, 3}$   corresponding to
{\em extended }  and 
nodal Burniat surfaces with $K^2_S = 3$
is an irreducible open set, 
normal, unirational
        of  dimension 4.
        
Moreover the 
base of the Kuranishi family of $S$ is smooth.}
\medskip

We are also 
almost done with the proof of our second main theorem

\medskip

{\bf 
Theorem \ref{path}}
{\em  The deformations of nodal Burniat surfaces 
with $K^2_S =4,3$ to extended  Burniat surfaces with $K^2_S =4,3$ 

yield examples where $\Def(S,(\ZZ/2\ZZ)^2) \ra \Def(X,(\ZZ/2\ZZ)^2)$ 
is not surjective.

Moreover, $\Def(S,(\ZZ/2\ZZ)^2) \subsetneq 
\Def(S)$, whereas for the canonical model we have: 
$\Def(X,(\ZZ/2\ZZ)^2) = \Def(X)$.

The moduli space of pairs, of an 
extended (or nodal) Burniat surface with $K^2_S =4,3$ and   a 
$(\ZZ/2\ZZ)^2$-action, 
is disconnected; but its image in the moduli 
space
is a connected open set.}

\begin{proof}
By propositions 
\ref{famiglia4} and \ref{famiglia3} we have two families with smooth 
connected rational base
of dimension 3, resp. 4, parametrizing all 
the canonical models $X$ of the surfaces in $\sN \sE \sB_4$, resp. 
$\sN \sE \sB_3$.

In the previous theorem \ref{main1} we showed that 
the base of the Kuranishi family of $S$  is smooth, hence
base change 
of these families yield the Kuranishi family of $S$. 

The above 
families of canonical models $X$ yield the Kuranishi family of $X$ 
e.g.  by the theorem of Burns and Wahl.

 Propositions 
\ref{famiglia4} and \ref{famiglia3}, exhibiting all the canonical 
models as bidouble covers of   normal Del Pezzo surfaces, 
 
immediately show that  $\Def(X,(\ZZ/2\ZZ)^2) = \Def(X)$.
 
 Let now 
$S$ be a nodal Burniat surface.
 
Since, by (7.1), page 23,  of 
\cite{cime88}  $\Def(S,(\ZZ/2\ZZ)^2) \subsetneq \Def(S)$ is the 
intersection with
 $H^1(\Theta_S)^0$, which is the  smooth subvariety 
corresponding to the nodal Burniat surfaces,
 we obtain that 
$\Def(S,(\ZZ/2\ZZ)^2) \subsetneq \Def(S)$.
 
 On the other hand, for 
instance in the case $K^2_S = 4$, we explicitly see that $\sN \sE 
\sB_4$ is the union of two
 families of bidouble covers, the family 
of nodal Burniat surfaces, respectively the family of extended 
Burniat surfaces:
 hence the moduli space of pairs $(S, 
(\ZZ/2\ZZ)^2)$ has exactly two connected 
components.

\end{proof}

\section{One parameter limits of extended 
Burniat
surfaces with $K_S^2 =4$}\label{degenerations}

In this 
section we  shall prove the following:

\begin{theo}\label{closed} 
The family of
extended  Burniat surfaces with $K_S^2 = 4$ 
yields,
together with the family of nodal Burniat surfaces
with 
$K_S^2 = 4$, a closed subset
    $\sN \sE \sB_4$     of the moduli 
space.
\end{theo}

This will be accomplished through the
 
study of limits of one parameter families of such extended
Burniat 
surfaces: we shall indeed show that only nodal Burniat 
surfaces (or 
extended Burniat surfaces) occur.

Let $Y'$ be a normal $\QQ$- 
Gorenstein surface and denote the
dualizing sheaf of $Y'$ by 
$\omega_{Y'}$.

Then there is a minimal natural number $m$ such 
that
$\omega_{Y'}^{\otimes m}$ is an invertible sheaf and  it
makes 
sense to define $\omega_{Y'}$ to be ample, respectively
anti-ample; 
$Y'$ is Gorenstein iff $ m=1$.

We recall the following results which 
were shown in \cite{burniat2}.
\begin{prop}\label{gorenstein}
 
Let $Y'$ be a normal $\QQ$-Gorenstein Del Pezzo surface 
(i.e.,
$\omega_{Y'}$ is anti-ample) with
         $K^2_{Y'} \geq 
4$.
Then $Y'$ is in fact 
Gorenstein.
\end{prop}

\begin{prop}\label{familydc}
         Let 
$T$ be a smooth affine curve, $t_0 \in T$, and let $f 
\colon
\mathcal{X} \ra T$ be a flat  family of
canonical surfaces. 
Suppose that $\X_t$ is the canonical model of a
Burniat surface with 
$4
\leq K_{\mathcal{X}_t}^2 $ for $t \neq t_0 \in T$. Then there is a 
biregular
action of  $G: = (\ZZ / 2\ZZ)^2$ on $\X$ yielding a one 
parameter
family of finite $(\ZZ / 2 \ZZ)^2$-covers
\begin{equation*}
\xymatrix{ 
\mathcal{X}\ar[dr]_f\ar[rr]&& \mathcal{Y} \ar[dl] \\ & T&,& 
}
\end{equation*} (i.e., $\X_t \ra \mathcal{Y}_t$ is a finite $(\ZZ / 2\ZZ)^2$-cover), 
such that
$\mathcal{Y}_t$ is a Gorenstein Del Pezzo 
surface for each $t \in T$.

\end{prop}
Observe that the above result remains 
true if we replace ``Burniat
surface'' by ``extended Burniat 
surface''.

This implies immediately the following:

\begin{cor}
 
Consider a one parameter family of bidouble covers $\mathcal{X} 
\ra
\mathcal{Y}$ as  in prop. \ref{familydc}.
Then the branch locus 
of $\mathcal{X}_{t_0} \ra \mathcal{Y}_{t_0}$ is
the limit of the 
branch locus of
$\mathcal{X}_t
\ra \mathcal{Y}_t$, and it is 
reduced.
\end{cor}

Note that the limit of a line on the del Pezzo 
surfaces
$\mathcal{Y}_t$ is a line on the del Pezzo
surface 
$\mathcal{Y}_{t_0}$, and, as a consequence of the above assertion,
 
two lines in the branch locus in
$\mathcal{Y}_t$ cannot tend to the 
same line in
$\mathcal{Y}_{t_0}$.

\begin{rem} Let $X$ be the 
canonical model of an extended 
Burniat surface with
$K_X^2 
=4$.
Recall that $X$ is smooth
for a general member of the family of 
extended Burniat surfaces, whereas
$X$ has one
ordinary node if $X$ 
is the canonical model of a nodal Burniat
surface with $K^2 = 4$.

In 
the extended case the branch locus consists of the union of
$3$ 
hyperplane sections, containing   $8$
lines, $2$ conics and the node. 
In the nodal Burniat case 
one of the conics degenerates to two 
lines, hence the branch
locus  consists instead of $10$ lines and one 
conic.
\end{rem}

The first step towards proving theorem \ref{closed} 
is the following:
\begin{prop}\label{onenode}
       Consider a one 
parameter family of bidouble covers $\mathcal{X} \ra
\mathcal{Y}$ as 
in prop. \ref{familydc} except that  $\sX_t$ is an extended
 Burniat 
surface with $K_{X_t}^2 = 4$ for $t \neq 0$.

Then $\mathcal{Y}_0$ is 
a normal Del Pezzo surface with exactly one node
as 
singularity.
\end{prop}

\begin{lemma}
A normal singular Del Pezzo 
surface with
$K_{\sY_0}^2 =4$ containing at least 8 lines has as 
singularities either 
\begin{enumerate}
\item
one node, and then it 
contains 12 lines, or
\item
 two nodes, and then it contains 9 lines, 
or

\item
an $A_2$ singularity, and then it contains 8 lines,
4 of 
which pass through the singular 
point.
\end{enumerate}

\end{lemma}
\begin{proof}
The assertion is a 
generalization of Proposition 3.6 of \cite{burniat2}, page 581,
see 
especially the proof in the appendix ibidem, pages 585-587.

We blow 
up $r=5$ points in the plane.

By the estimate about the loss of 
number of lines when one has
a chain of $k$ infinitely near points, 
we see that $ k\geq 4$ implies that the number of lines
is less than 
$16- 11 = 5$.

If there is a chain with $ k=3$, the same estimate 
gives  a loss of $8$, and we cannot
then have other (-2)-curves, else 
the number would be strictly smaller than $ 16-8 = 8 $.

In this case 
we get an $A_2$ singularity and 8 lines.

In fact,  in the chosen 
plane model
we have 5 points lying on an irreducible conic $C$, of 
which $P_2$ infinitely near to $P_1$,
and $P_3$ infinitely near to 
$P_2$. The lines are given by
$$ E_3, E_4, E_5, | L - E_1 - E_4|,   | 
L - E_1 - E_5|,  | L - E_1 - E_2|,  | L - E_4 - E_5|, C',$$
where 
$C'$ is the strict transform of $C$. 

In this case the  4 lines 
passing through the singular point are 
$$ E_3,  | L - E_1 - E_4|, 
| L - E_1 - E_5|,  | L - E_1 - E_2| .$$

In the case where there is 
no chain of three infinitely near points by a standard Cremona 
transformation
as in \cite{burniat2}, ibidem,
we may reduce to the 
case where there are no infinitely near points and then we have 
 
that the weak Del Pezzo surface is $\hat{Y}_0 : = 
\hat{\PP}^2(P_1,
\ldots , P_5)$, where $P_1,P_2,P_3$ and $P_1, P_4, 
P_5$ are
collinear. 

Then $\hat{Y}_0$ contains nine lines. In fact, 
the set
of lines of $\hat{Y}_0$ is:
$$
\mathcal{L}:=\{ E_1, \ldots , 
E_5, L-E_2-E_4, L-E_2-E_5, 
L-E_3-E_4,
L-E_3-E_5\}.
$$

\end{proof}

\begin{proof}[Proof of prop. 
\ref{onenode}]
 Since the branch locus of $\mathcal{X}_t 
\ra
\mathcal{Y}_t$ contains eight lines for $t \neq 0$, also the 
branch
locus of $\mathcal{X}_0 \ra
\mathcal{Y}_0$ contains eight 
lines. 

We want to show that cases (2) and (3) of the previous lemma 
cannot occur.

We start by eliminating case (3).

Here,  the $A_2$ 
singularity must be a limit ofthe node of $\sY_t$, hence
the bidouble 
cover is branched at the singular point.

The bidouble cover is a 
RDP, hence, looking at table 2, page 90 of \cite{autRDP},
and table 
3, page 93 ibidem, 
we see that the branch locus is analytically 
isomorphic to
\begin{itemize}
\item
an ordinary cusp $ \{ y=0= z^2 + 
x^3=0\}$ for $E_6 = \{ z^2 + x^3 + t^4 = 0 \} \ra A_2 = \{ z^2 + x^3 
+ y^2 = 0 \}$,
\item
two lines $ \{ x=0= z^2 + y^2=0\}$ for  $A_5  = 
\{ z^2 + w^6 + y^2 = 0 \} \ra A_2 = \{ z^2 + x^3 + y^2 = 0 
\}$,
\item
two lines $ \{ x=0= z^2 + y^2=0\}$ for  the composition of 
$A_2 \ra A_5$ (ramified only at the singular point) with the 
previous
$A_5  = \{ z^2 + w^6 + y^2 = 0 \} \ra A_2 = \{ z^2 + x^3 + 
y^2 = 0 \}$.
\end{itemize}

We observe however that by our previous 
arguments the branch locus contains the 8 lines, 4 of which 
pass
through the $A_2$ singularity, contradicting the above local 
description of the branch locus.

Assume now by contradiction that we 
have case (2), i.e., $\sY_0$ has two nodes. Then

\begin{claim}
 
$E_1$ is not a component of the total branch locus $\Delta$ of 
$\hat{X}_0 \ra
\hat{Y}_0$ , i.e.,
$$
       E_2, \ldots , E_5, 
L-E_2-E_4, L-E_2-E_5, L-E_3-E_4, L-E_3-E_5
$$
are exactly the 8 lines 
contained in $\Delta$.
\end{claim}
\begin{proof}[Proof of the 
claim.]
Assume that $E_1$ is contained in the total branch locus 
$\Delta$ of the bidouble
cover $\hat{X}_0 \ra
\hat{Y}_0$. Then 
$\Delta$ contains three lines intersecting
one of the two $(-2)$ 
curves. But a bidouble cover of a node branched
in at least three 
lines does not give a rational double point, as shown by the 
classification
recalled in section \ref{locdefnode}. A 
contradiction.
\end{proof}

Since for each node  $\nu_1, \nu_2$ there 
are two lines in the total branch divisor passing through
$\nu_i$, it 
follows by the classification given  in section \ref{locdefnode}, 
that $N_1, N_2 \leq \Delta$ and that $(\Delta-N_i)N_i =2$.

Denote by 
$\pi \colon \hat{Y}_0 \ra Y'$ be the desingularization map. 

Then 
$\pi_*(\Delta) \equiv -3K_{Y'}$, whence 
$$
\Delta \equiv 
-3K_{\hat{Y}_0} + n_1N_1+n_2N_2.
$$
Then $2=(\Delta-N_i)N_i 
=(n_i-1)N_i^2 = 2(1-n_i) \ \Leftrightarrow \ n_i=0$.

We conclude 
that

$$
\Delta \equiv -3K_{\hat{Y}_0}.
$$

Observe 
that
$$
-3K_{\hat{Y}_0} - \sum_{l \in \mathcal{L} \setminus \{E_1\}}l 
- N_1 - N_2  \equiv
3L - E_1-E_2- \ldots -E_5.
$$

Since no other 
component of $\Delta$ can intersect the
$(-2)$-curves, we see 
immediately that the remaining two
components of $\Delta$ are: 
$$L-E_1, 2L-E_2 -E_3-E_4-E_5.$$

We write now
$$
       \Delta_1 = 
\lambda_1L - E_1 - a_2E_2-a_3E_3-a_4E_4-a_5E_5,
$$
$$
\Delta_2 = 
\lambda_2L - E_1 - b_2E_2-b_3E_3-b_4E_4-b_5E_5,
$$
$$
\Delta_3 = 
\lambda_3L - E_1 - c_2E_2-c_3E_3-c_4E_4-c_5E_5.
$$

Here we have used 
that, since $E_1$ is not a
component of $\Delta$ and since  $\Delta_i 
+ \Delta_j$ has to be divisible by
two, the only possibility is $$E_1 
\cdot (\Delta_1,\Delta_2,\Delta_3) = (1,1,1).$$

Note that, since 
$\lambda_1 + \lambda_2 + \lambda_3 = 9$ (and again since
$\Delta_i + 
\Delta_j$ is divisible by two) we 
have:
$$
(\lambda_1,\lambda_2,\lambda_3) \in 
\{(3,3,3),(1,3,5),(1,1,7)\}.
$$
Moreover, since the branch divisor is 
reduced,  for each $i$ it happens that, 
 among the three numbers 
$a_i,b_i,c_i $, there cannot be two which are negative, and 
 if one 
such a number is negative, then it is $= \ -1$; hence
the only 
possibilities are:

$\{a_i,b_i,c_i\} =\{1,1,1\}$ or $\{-1,1,3\}$, for 
$i \in
\{2, \ldots , 
5\}$.

\bigskip
\noindent
$\underline{(\lambda_1,\lambda_2,\lambda_3)=(3,3,3):}$
then 
we get for the character sheaves:
$$
\mathcal{L}_1 = \hol(3L -E_1 - 
\frac{b_2+c_2}{2}E_2-
\frac{b_3+c_3}{2}E_3- \frac{b_4+c_4}{2}E_4- 
\frac{b_5+c_5}{2}E_5),
$$
$$
\mathcal{L}_2 = \hol(3L -E_1 - 
\frac{a_2+c_2}{2}E_2-
\frac{a_3+c_3}{2}E_3- \frac{a_4+c_4}{2}E_4- 
\frac{a_5+c_5}{2}E_5),
$$
$$
\mathcal{L}_3 = \hol(3L -E_1 - 
\frac{a_2+b_2}{2}E_2-
\frac{a_3+b_3}{2}E_3- \frac{a_4+b_4}{2}E_4- 
\frac{a_5+b_5}{2}E_5).
$$

\medskip
Note that $(a_i,b_i,c_i) 
=(1,1,1)$ for all $i \in \{2, \ldots ,5\}$
implies that $p_g(\mathcal{X}_0) 
\neq 0$, whence w.l.o.g. 
$$(a_2,b_2,c_2)
= (-1,1,3).$$
 Then $E_2 
\leq \Delta_1$, and by the local calculations in
section 
\ref{locdefnode} this implies that also $E_3 \leq \Delta_1$ 
(since
the two lines of the branch locus intersecting a $(-2)$-curve 
belong
to the same $\Delta_i$). Therefore 
$$(a_3,b_3,c_3) \in 
\{(-1,*,*),
(1,1,1)\}.$$

Again using $p_g(\mathcal{X}_0) = 0$, we 
conclude (looking at
$\mathcal{L}_3$) that (up to exchanging $P_4$ 
with $P_5$)
$$(a_4,b_4,c_4) \in \{(3,1,-1), (1,3,-1)\},$$ 
and
again 
this implies that
$$(a_5,b_5,c_5) \in \{(*,*,-1), (1,1,1)\}.$$

But 
in all of these cases we have $$\frac{a_i + c_i}{2} \in \{0,1\} \ 
\forall \in \{2, \ldots , 5\},$$ contradicting $p_g = 
0$.

\medskip
\noindent
$\underline{(\lambda_1,\lambda_2,\lambda_3)=(1,3,5):}$
here 
we have
$$
\mathcal{L}_1 = \hol(4L -E_1 - 
\frac{b_2+c_2}{2}E_2-
\frac{b_3+c_3}{2}E_3- \frac{b_4+c_4}{2}E_4- 
\frac{b_5+c_5}{2}E_5),
$$
$$
\mathcal{L}_2 = \hol(3L -E_1 - 
\frac{a_2+c_2}{2}E_2-
\frac{a_3+c_3}{2}E_3- \frac{a_4+c_4}{2}E_4- 
\frac{a_5+c_5}{2}E_5),
$$
$$
\mathcal{L}_3 = \hol(2L -E_1 - 
\frac{a_2+b_2}{2}E_2-
\frac{a_3+b_3}{2}E_3- \frac{a_4+b_4}{2}E_4- 
\frac{a_5+b_5}{2}E_5).
$$
Again, $p_g = 0$ implies that there is an 
$i \in \{2, \ldots , 5\}$
such that $\frac{a_i +c_i}{2} =2$. W.l.o.g. 
we can assume that
$\frac{a_2+c_2}{2} =2$. Therefore 
$$(a_2,b_2,c_2) 
\in \{(3,-1,1),
(1,-1,3)\},$$
 whence 
$$(a_3,b_3,c_3) \in 
\{(3,-1,1), (1,-1,3),
(1,1,1)\}.$$
 Then $\frac{b_2+c_2}{2}, \ 
\frac{b_3+c_3}{2} \leq 1$ and
$\frac{b_4+c_4}{2}, \ \frac{b_5+c_5}{2} 
\leq 2$, which implies that
$\hol(L-E_2-E_4) \subset 
\hol(K_{\hat{Y}_0}) \otimes \mathcal{L}_1$,
contradicting 
$p_g(\mathcal{X}_0) = 
0$.

\bigskip
\noindent
$\underline{(\lambda_1,\lambda_2,\lambda_3)=(1,1,7):}$
this 
case can be excluded since 
$$4 = \De_3 \cdot  (-  K_{\hat{Y}_0}) = 3 
\la_3 - 1 - \sum_{i=2}^5 c_i  \Rightarrow  12 \geq \sum_{i=2}^5 c_i = 
16,$$
a contradiction.

This proves the 
proposition.
\end{proof}

Consider a one parameter family of bidouble 
covers $\mathcal{X} \ra
\mathcal{Y}$ as  in prop. \ref{onenode}.
Then 
$Y' :=\mathcal{Y}_0$ is a normal Del Pezzo surface with exactly
one 
node. 

Let $\tilde{Y}$ be the blow up of $Y'$ in the node and
denote 
the exceptional $(-2)$-curve of $\tilde{Y}$ over the node 
by
$A$.

The following result concludes the proof of theorem 
\ref{closed}.

\begin{prop}
For the limit of a one parameter family 
of extended  Burniat
surfaces with $K_S^2 =4$ we 
have:
\begin{enumerate}
\item if $A$ does not intersect 
$\Delta-A$, 
then $\mathcal{X}_0$ is an extended  Burniat surface with $K_S^2 =4$ 
;
\item if $A$ intersects $\Delta-A$, then
$\mathcal{X}_0$ is a nodal 
Burniat surface with $K_S^2 =4$ 
.
\end{enumerate}

\end{prop}

\begin{proof}
We can assume that 
$\tilde{Y} = \hat{\PP}^2(P_1, \ldots , P_5)$, and w.l.o.g. 
$P_1,P_4,P_5$ collinear, i.e.,
$A \equiv L-E_1-E_4-E_5$. 

Recall 
that we have shown that in both cases $A$ is contained in the branch 
locus,
hence the two alternatives are that $A$ is a connected 
component of the branch locus, or not.

1) In the first case, 
argueing as in proposition \ref{onenode}, we get that  the total 
branch locus is $\Delta \equiv
-3K_{\tilde{Y}} +A$. 

It is easy to 
see that $\tilde{Y}$ contains exactly
8 lines $l_1, \ldots , l_8$ 
which do not intersect $A$. Then these 8 lines
have to be contained 
in  $\Delta$. 

Then $\Delta - A- \sum_{i=1}^8 l_i
\equiv 3L - \sum 
E_i$, which has to split into two Del Pezzo conics,
which then have 
to be $L-E_1$ and $2L - \sum _{i=2}^5 E_i$. Hence we
get an extended 
Burniat surface.

\noindent
2)
Here $ L-E_1 -E_4-E_5 \equiv A \leq 
\Delta \equiv -K_{\tilde{Y}}$. 

Observe that
$\tilde{Y}$ contains 
exactly 4 lines intersecting $A$: $E_1,E_4,E_5,
L-E_2-E_3$. By our 
local calculations in section \ref{locdefnode} two
of these four 
lines are components of the total branch divisor and
the  two other 
not.

W.l.o.g. we can assume $E_1, L-E_2-E_3 \leq \Delta$. Since 
$E_4$ and
$E_5$ are not contained in the branch divisor, we see 
(writing $\Delta_i$
as in the proof of proposition \ref{onenode}) that$(a_4,b_4,c_4) 
=
(a_5,b_5,c_5) = (1,1,1)$.

Now it is straightforward that 
$(\lambda_1, \lambda_2, \lambda_3) = (3,3,3)$
(use the same argument 
as in the proof of prop. \ref{onenode} to
exclude the cases $(1,3,5)$ 
and $(1,1,7)$).

 Since $p_g= 0$, we have (up to a permutation of 
\{1,2,3\})
$$
\frac{b_1+c_1}{2} =\frac{a_2+c_2}{2}= \frac{a_3+b_3}{2} 
= 2.
$$

W.l.o.g. we can assume $(a_1,b_1,c_1) = (-1,1,3)$; then 
$E_1, L-E_2-E_3
\leq \Delta_1$. 

Therefore 
$$(a_2,b_2,c_2) \in 
\{(3,-1,1), (1,-1,3)\}$$
 and
$$(a_3,b_3,c_3) \in \{(3,1,-1), 
(1,3,-1)\}.$$

 But only $(a_2,b_2,c_2) =
(3,-1,1)$ and 
$(a_3,b_3,c_3) = (1,3,-1)$ is possible (since a cubic
cannot have two 
triple points, i.e., this would contradict the
effectivity of 
$\Delta_i$ for some $i$). 

Therefore we get a nodal
Burniat 
surface.
\end{proof}

\section{Nodal and extended Burniat surfaces do 
not form a closed set for $K^2_S = 3$}

We are going to exhibit 
surfaces which are in the closure of the family of nodal and extended 
Burniat surfaces, but for which the image 
of the bicanonical map is 
a normal cubic with other singularities than 3 nodes.

In our first 
example we exhibit a 3 -dimensional family with a 4-nodal cubic as 
image.

Consider a specialization of the 6 points $P_1, \dots , P_6$ 
in $\PP^2$ so that $P_1, P_2, P_3$ become collinear,
and, more 
precisely, the point $P_2$ moves in the line joining $P_4$ and $P_6$ 
till it reaches the line joining $P_1$ and $P_3$.

Then $P_1, \dots , 
P_6$ are the vertices of a complete quadrilateral with sides  $N_1, 
N_2, N_3, N_4$:
here we identify $N_4$ to the (-2) curve $N_4 \equiv 
L - E_1 - E_2 - E_3$ on the weak Del Pezzo $\tilde{Y}$ of degree 
3
obtained blowing up the 6 points. Our notation for $N_1, N_2, N_3$ 
remains the same, and $\tilde{Y}$ 
is the minimal resolution of the 
4-nodal cubic surface $Y' : = \Sigma$.

We consider exactly the same 
divisors as the strictly extended Burniat divisors in 4) of 
definition \ref{df}.
We obtain a three dimensional family of bidouble 
covers $X$ of  $\Sigma$, with total branch locus
consisting of 9 
connected components, namely:
$$ N_1, N_2, N_3 ; \  \Ga_1, \Ga_2, 
\Ga_3 ; \  G_1, G_2 , G_3. $$

$G_1, G_2 , G_3$ correspond to the 
three diagonals of the quadrilateral, and are the 3 lines of 
$\Sigma$
not passing through the nodes, whereas $\Ga_1, \Ga_2, \Ga_3$ 
are conics as in definition \ref{df}.
The canonical models $X$ have 
therefore 4 nodes lying over the node of $\Sigma$ corresponding to 
$N_4$.

We have therefore proven:
\begin{prop}
The closure of the 
(4-dimensional) open set corresponding to nodal and extended Burniat 
surfaces with  $K^2_X = 3$
contains a 3-dimensional family of 
canonical models which are bidouble covers of a 4-nodal cubic surface 
$\Sigma$.

Each such surface $X$ has 4 nodes, lying over one fixed 
node of $\Sigma$, and  where the bicanonical map 
$\Phi_2 \colon X 
\ra \Sigma$ is unramified.

\end{prop}

In our second example we 
obtain a 3-dimensional family of bidouble covers of a cubic surface 
$Y'$ with a singularity of type $D_4$.

We give this example using 
the different planar realization which was indeed the way we found 
our first description 
of the deformation of nodal Burniat surfaces 
with $K^2_S = 3$ to extended Burniat surfaces.

To do this, we 
relabel the 6 points in the plane as follows: 
$$P'_3 : = P_4 , \ 
P'_2 : = P_5 , \ P'_1 : = P_6 .$$

We have therefore irreducible 
rational curves 
\begin{multline*} D_{i,1} : = L - E_i - E_{i+1}, \  D_{i,2} : =N_i 
=  L - E_i - E'_{i+1}- E'_{i+2}, \\  
D_{i,3} : =G_i =  L - E_i - 
E'_{i}
\end{multline*}

on the weak Del Pezzo $\hat{Y}$ of degree 3.

Blowing down 
the 3 (-1) curves $D_{i,1}$ ($i=1,2,3$) first, and then the strict 
transform of the 3 (-2) curves
$D_{i,2}$ ( $i=1,2,3$) we obtain 
another copy of the projective plane where one has blown up
three 
points $Q_i$ ($i=1,2,3$) and three points $Q'_i$ ($i=1,2,3$), where 
$Q'_{i}$ is infinitely near to $Q_i$.

We denote by slight abuse of 
notation by $Q_i$ the full transform of the point $Q_i$, namely, the 
divisor
$D_{i,1} + D_{i-1,2}$, and by $Q'_i$ the full transform of 
the point $Q'_i$,namely, the divisor
$D_{i,1}$. 

The pull back of 
the system of lines in the new $\PP^2$ is, by the Hurwitz formula, 
the linear system 
$$ \sL : = 4 L - 2 \sum_i E_i -   \sum_i E'_i.$$

And the curve $D_{i,3}$ is linearly equivalent to 
$$D_{i,3} \equiv 
\sL - 2 D_{i+1,1}- D_{i,2} = \sL - Q_{i+1}- Q'_{i+1}. $$

Hence 
$D_{i-1,2} = Q_i - Q'_i$, and we can write the branch loci for the 
extended Burniat surfaces
as:
$$\De_i \in  D_{i,3} +  D_{i+1,2} + |D_{i,3} + D_{i+1,3} | = D_{i,3} 
+  |Q_{i-1}- Q'_{i-1} | + |D_{i,3} + D_{i+1,3} |=$$
$$ = | \sL - Q_{i+1}- Q'_{i+1} | +  |Q_{i-1}- Q'_{i-1} | +  | 2  \sL 
- Q_{i+1}- Q'_{i+1} - Q_{i-1}- Q'_{i-1}| =$$
$$ = | \sL - Q_{i+1}- Q'_{i+1} | +  N_{i-1}  +  | 2  \sL - Q_{i+1}- 
Q'_{i+1} - Q_{i-1}- Q'_{i-1}| .$$

Now, we simply let the three points $Q_1, Q_2, Q_3$ become collinear, 
but we let the tangent directions
$Q'_i$ remain general.

The blow up of the plane in the 6 points possesses now 4 (-2) curves, 
the three curves $N_1, N_2, N_3$
and  the strict transform $N$ of the line through $Q_1, Q_2, Q_3$. 
Since $N$ intersects each $N_i$ and these are disjoint,
the corresponding normal Del Pezzo surface $Y'$ has a singularity of 
type $D_4$.

Letting the branch divisor be as before (namely, take pull backs of 
general conics in $ | 2  \sL - Q_{i+1}- Q'_{i+1} - Q_{i-1}- Q'_{i-1}| 
$),
we obtain

\begin{prop}
The closure of the (4-dimensional) open set corresponding to nodal 
and extended Burniat surfaces with  $K^2_X = 3$
contains a 3-dimensional family of canonical models which are 
bidouble covers of a normal cubic surface $Y'$
with a singularity of type $D_4$.

The branch locus on $Y'$ has the singular point as an isolated point, 
and the local covering is determined
by the epimorphism  $D_4 \ra (\ZZ/ 2 \ZZ )^2 =( D_4)^{ab}$ of the 
local fundamental group of the singularity to its
abelianization.

\end{prop}
\begin{proof}
The inverse image of the  (-2) curves in the bidouble cover are: the 
inverse image $N'$ of $N$, which is a (-8) curve,
and, for each $N_i$, there is a pair of (-1) curves meeting  $N'$. 
After contracting the 6 (-1) curves  we obtain
a (-2) curves.

\end{proof}

{\em Acknowledgements:} Thanks to Stephen Coughlan for writing a MAGMA script
in order to verify the calculations of proposition \ref{dimensions}.


\smallskip
\noindent {\bf Authors' Addresses:}\\
\noindent I.Bauer, F. Catanese, \\ Lehrstuhl Mathematik VIII,\\
Mathematisches Institut der Universit\"at
Bayreuth\\ NW II,  Universit\"atsstr. 30\\ 95447 Bayreuth\\
          email: ingrid.bauer@uni-bayreuth.de,
                 fabrizio.catanese@uni-bayreuth.de

\section{Appendix: an alternative proof of statements 1), 2), 3) of 
proposition \ref{dimensions}}

In this appendix we present other methods to calculate the space of sections
of twisted logarithmic sheaves, in particular  a fibration method.

Assume that we have $d$ smooth rational curves $C_\alpha \subset Y$
contained in a smooth
algebraic surface $Y$, meeting with distinct tangents in a point $O$,
a divisor $B_\alpha$ on $C_\alpha$ of degree $0$, $1$ or
$2$, and disjoint from
$O$, and let $Z$ be the blow up of $Y$ in the point $O$.
Denote by $D_\alpha$ the strict transform of $C_\alpha$,
and denote by $\Omega_{Y}^1 ((\log C_{\alpha}(-B_{\alpha}))_{\alpha \in A}) $
the sheaf which is the inverse image, under the residue sequence, of
$\oplus_{\alpha \in A} \hol_{C_{\alpha}}(-B_{\alpha}) $.

Then by 4) of  proposition \ref{Hauptlemma} we have an exact sequence
\begin{multline*}
    0 \ra  \Omega_{Y}^1 ((\log C_{\alpha}(-B_{\alpha}))_{\alpha \in A}) \ra  \\
\ra p_*\Omega^1_{Z}( (\log D_{\alpha}(-B_{\alpha}))_{\alpha
\in A})(E)
\ra \CC^{d-2}_O \ra 0
\end{multline*}

which is exact on global sections if
$$ h: = dim_{\CC}H^1 (\Omega_{Y}^1 ((\log
C_{\alpha}(-B_{\alpha}))_{\alpha \in A})) = 0.$$
Or, more generally, iff $ h = h'$, where
$$ h': = dim_{\CC}H^1 (\Omega_{Z}^1 ((\log
D_{\alpha}(-B_{\alpha}))_{\alpha \in A})) .$$
Consider the exact sequence
$$ 0 \ra \Omega_{Y}^1 \ra \Omega_{Y}^1 ((\log
C_{\alpha}(-B_{\alpha}))_{\alpha \in A}) \ra
\bigoplus_{\alpha=1}^d \hol_{C_{\alpha}}(-B_{\alpha}) \ra 0,
$$
and assume that $H^2(\Omega_{Y}^1) = 0$.

Then $ h = a + b $, where $ a $ is the number of $\alpha$'s such that
$B_\alpha$ has degree
$2$, while $b$ is the difference of the dimensions between $H^1(\Omega_{Y}^1) $
and the subspace generated by the Chern classes of the $C_{\alpha}$'s such that
$B_\alpha$ has degree
$0$. If  we choose $Y = \PP^2$ then $h=0$ as soon as no $B_\alpha$ has degree
$2$, and some $B_\alpha$ has degree
$0$.

Otherwise, one can calculate $h'$ in a similar way. We assume for
simplicity that
$Y = \PP^2$. We have a similar exact sequence
$$ 0 \ra \Omega_{Z}^1 (E) \ra \Omega_{Z}^1 ((\log
D_{\alpha}(-B_{\alpha}))_{\alpha \in A})(E)
\ra  \bigoplus_{\alpha=1}^d \hol_{C_{\alpha}}(O -B_{\alpha}) \ra 0,
$$
and since $H^1 (\hol_{C_{\alpha}}(O -B_{\alpha}) )= 0$
by our assumption, we get that $h'$ is the dimension of the cokernel of
      $$\bigoplus_{\alpha=1}^d H^0 (\hol_{C_{\alpha}}(O -B_{\alpha}) )
\ra H^1 (\Omega_{Z}^1 (E) ).$$
To calculate the last space, observe that
$$\Omega_{Z}^1 \otimes \hol_E = \hol_E(-2) \oplus \hol_E(1) $$
whence $h^1 (\Omega_{Z}^1 (E))= h^1 (\Omega_{Z}^1) + 1 = 3.$

These  criteria can now be used in order to prove statements 1), 2), 
3) of proposition \ref{dimensions}.

We can  prove 1) and 2) simultaneously for $i=1$.

Observe that $D_1 = \De_1 + N_1$, that
$ \Lambda_1 = L_1 + N_1$, and apply  Lemma \ref{poles}
in order to conclude that
    $$H^0 ( \Omega_{\tilde{Y}}^1 (\log (\De_1))(E_1 - E_{3} + N_1 )) 
\cong H^0 ( \Omega_{\tilde{Y}}^1 (\log (D_1))(E_1 - E_{3} )).$$

By Lemma \ref{Hauptlemma} we can blow down $E_3$ and obtain
$H^0 ( \Omega_{\tilde{Y'}}^1 (\log (D'_1))(E_1)).$
In this case the respective degrees of the divisors $B_{\alpha}$ are
$0,1,2$ hence $h=1$.
We have to decide whether  $h'$ is  $0$ or $1$.
We contract  $E_3, E_4, E_5$, and we let $Z$ be the blow up of the
plane in $P_1$.
We must calculate $h^0 ( \Omega_{Z}^1 (\log (C_{\alpha} ( -
B_{\alpha}))) (E_1)).$
Here the curves $C_{\alpha} $ are fibres of the ruling of $Z$, $f
\colon Z \ra  \PP^1$.
Using the exact sequence
$$(**)  \ 0 \ra f^* \Omega_{\PP^1}^1 \ra   \Omega_{Z}^1 \ra \omega_{Z|\PP^1}=
\hol_Z (- F - 2 E_1) \ra 0$$
we obtain the analogous sequence
$$0 \ra f^* \hol_{\PP^1}(1)(E_1) \ra   \Omega_{Z}^1 (\log (C_{\alpha}
)(E_1) \ra  \hol_Z (- F -
E_1) \ra 0$$
to infer that
$$H^0(f^* \hol_{\PP^1}(1) (E_1))= H^0 (  \Omega_{Z}^1 (\log
(C_{\alpha} ))(E_1)).$$
We are imposing some vanishing on three points
lying in two fibres, hence we get the sections of
$H^0( \hol_Z ( F + E_1))= p^* H^0( \hol_{\PP^2}( 1))$ vanishing in the
three points  $P_4, P_5, P_2$,
whence we conclude that this space  has dimension  $=0$.

This argument  shows 1) also for $i=2,3$.

For a nodal Burniat with $m=2$  the space
$H^0 ( \Omega_{\tilde{Y}}^1 (\log (D_i))(E_i - E_{i+2} ))$,
vanishes for $i=2$, but it has dimension equal to $1$ for
$i=3$, since then the three points  $P_4, P_5, P_1$ are collinear.

Let's proceed with 2).

     For $i=2,3$ $$ H^0 ( \Omega_{\tilde{Y}}^1 (\log \De_i) (K_{\tilde{Y}} +
\Lambda_i  )) = H^0 ( \Omega_{\tilde{Y}}^1 (\log \De_i) (E_i - E_{i+2} )).$$
By applying again Lemma \ref{Hauptlemma}  for $i=3$ we can blow down
the curve $E_2$
and the curve $E_3$ and apply the residue sequence
to the sheaf $ \Omega_{\tilde{Y'}}^1 (\log \De'_3)$. Since each component
     is smooth and rational, we find that
$H^0 ( \Omega_{\tilde{Y'}}^1 (\log \De'_3)) = ker ( \CC^4 \ra H^1
(\Omega_{\tilde{Y'}}^1 ))$, while
$H^1 ( \Omega_{\tilde{Y'}}^1 (\log \De'_3)) = Coker ( \CC^4 \ra H^1
(\Omega_{\tilde{Y'}}^1 ))$.

The map is given by the Chern classes of $ L-E_1,   L-E_4, L-E_5,
L-E_1-E_4 - E_5.$
These generate a rank $4$ subspace of the 4-dimensional space $H^1
(\Omega_{\tilde{Y'}}^1)$
(we are blowing up 3 points in the plane),
     whence $h^0 ( \Omega_{\tilde{Y'}}^1 (\log \De'_3)) = 0 ,
     h^1 ( \Omega_{\tilde{Y'}}^1 (\log \De'_3)) = 0$.

We conclude, by the exact cohomology sequence associated to
$$ 0 \ra  \Omega_{\tilde{Y'}}^1 (\log \De'_3) \ra f_* (
\Omega_{\tilde{Y}}^1 (\log \De_3)
(K_{\tilde{Y}} + \Lambda_3  )) \ra \CC_{P_3} \ra 0,$$
    that $ h^0 ( \Omega_{\tilde{Y}}^1 (\log \De_3)
(K_{\tilde{Y}} + \Lambda_3  )) = 1 +h^0 ( \Omega_{\tilde{Y'}}^1 (\log
\De'_3)) =1.$

For  the case $i=2$ recall that $\De_2 \in  | L- E_2 - E_4| +  | L-
E_2 - E_5| +
    | 2L- E_2 - E_3- E_4 - E_5|$ consists of three smooth connected components.

    Blow down $ E_1,  E_3,  E_4 , E_5$ and obtain the ruled surface $Z $
equal to the
    blow up of the plane in $P_2$. Denote by $ f : Z \ra \PP^1$ the
standard fibration.

    The direct image $\De_2' : = f_* \De_2$  decomposes as the union of
two fibres $F_4$ and
    $F_5$ and a section $C$ with $C \cdot E_2 = 1$.

    We have to calculate the space of global sections of
    $$\sF : = \mathfrak M_{P_1} \Omega^1_Z ( \log F_4, \log F_5, \log
C(-P_3))(E_2)$$

   satisfying two linear conditions imposed by the points $P_4, P_5$.

    Using the exact sequence (**) we get the exact sequence

    $$0 \ra \mathfrak M_{P_1}  \hol_{Z}(E_2) \ra   \sF \ra  \mathfrak
M_{P_3} \mathfrak M_{P_1}
    \hol_Z (- F -
E_2 + C) \ra 0.$$

Observe that $ \hol_Z (- F -
E_2 + C) $ has degree 0 on each fibre and degree 1 on $E_2$. If $ D
\equiv - F -
E_2 + C \equiv L - E_2$ is effective, then $D$ is a fibre. Since no
fibre contains both $P_1, P_3$,
we obtain $$H^0 ( \mathfrak M_{P_3} \mathfrak M_{P_1}
    \hol_Z (- F -
E_2 + C) )= 0.$$
Since $|E_2|$ consists of the curve $E_2$, which does not contain
$P_1$, we conclude
that $ H^0 ( M_{P_1}  \hol_{Z}(E_2)) = H^0 (   \sF) = 0$.

To prove 3), by symmetry, we may assume without loss of generality that $i=1$.

Blow down all the curves $E_j$ excet $E_1$, so that , as usual, we
have the blow up $Z$
of the plane in a point ($P_1$) and  the standard fibration $ f : Z \ra \PP^1$.

By Lemma \ref{Hauptlemma} and since $E_3$ is a connected component of $D_1$,
the direct image $\sF$ of $ \Omega_{\tilde{Y}}^1 (\log (D_1) (E_1
- E_{3} )  $ is contained in  $ \Omega_{Z}^1 (\log (F_2 + F_6 +
F_{4,5})) (E_1 )   $
where $F_j$ denotes the unique fibre of $f$ passing through the point $P_j$.

More precisely, we have an exact sequence

$$0 \ra   \mathfrak M_{P_2} \mathfrak M_{P_4} \mathfrak M_{P_5}
\mathfrak M_{P_6} \hol_{Z}( F + E_1) \ra   \ \sF  \ra
    \hol_Z (- F -
E_1) \ra 0.$$

Clearly $H^0 ( \hol_Z (- F -
E_1)) = 0$ since $ F \cdot ( F + E_1) = 1$. On the other hand  $H^0 (
\hol_{Z}( F + E_1)) =  H^0 ( \hol_{\PP^2}(1))$,
hence the fact that the points $P_2, P_4,P_5,P_6$ are not collinear
implies the desired vanishing
$$ H^0 ( \mathfrak M_{P_2} \mathfrak M_{P_4} \mathfrak M_{P_5}
\mathfrak M_{P_6} \hol_{Z}( F + E_1))=0.$$

Thus 3) is proven.

\end{document}